\documentclass[reqno, a4paper, 10pt]{amsart}
\usepackage{amssymb,url, color, mathrsfs, multirow, makecell, array, caption}
\usepackage[colorlinks=true, bookmarks=true, pdfstartview=FitH, pagebackref=true, linktocpage=true]{hyperref}
\usepackage[short,nodayofweek]{datetime}
\usepackage[pagewise,running,mathlines,displaymath, switch]{lineno}
\usepackage{times}
\usepackage{indentfirst, tabularx}
\usepackage[table]{xcolor}
\usepackage{float, hhline, colortbl}
\usepackage{graphicx}
\usepackage{longtable}
\usepackage{empheq}

\usepackage[framemethod=TikZ]{mdframed}
 
\surroundwithmdframed[ topline=false, rightline=true, bottomline=false, leftmargin=\parindent,skipabove=\smallskipamount, skipbelow=\smallskipamount]{theorem}

\theoremstyle{plain}
\numberwithin{equation}{section}
\newtheorem{theorem}{Theorem}[section]
\newtheorem{lemma}[theorem]{Lemma}
\newtheorem{corollary}[theorem]{Corollary}
\newtheorem{proposition}[theorem]{Proposition}
\theoremstyle{remark}

\newtheorem{remark}[theorem]{Remark}

\newcommand{\N}{{\mathsf N}_{n,\lambda,p}^{k, \beta} }
\newcommand{\M}{{\mathsf M}_{n,\lambda,p}^{k,\alpha, \beta} }
\renewcommand{\P}{{\mathsf P}_{n,\lambda,p}^{\beta} }
\DeclareMathOperator{\R}{\mathbf{R}}

\newcommand{\EE}{\mathcal{E}_{\lambda, k}^{ \beta}}

\newcommand{\RR}{\mathcal{R}_{\lambda, k}^{ \beta}}
\DeclareMathOperator{\F}{\mathbf{F}}

\def\R{\mathbf R}

\numberwithin{equation}{section}

\makeatletter
\def\@cite#1#2{[\textbf{#1}\if@tempswa, #2\fi]}
\makeatother

\title[An \textit{optimal} HLS inequality on $\R^{n-k} \times \R^n$ ]{An \textit{optimal} Hardy--Littlewood--Sobolev inequality on $\R^{n-k} \times \R^n$ and its consequences}

\def\cfac#1{\ifmmode\setbox7\hbox{$\accent"5E#1$}\else\setbox7\hbox{\accent"5E#1}\penalty 10000\relax\fi\raise 1\ht7\hbox{\lower1.05ex\hbox to 1\wd7{\hss\accent"13\hss}}\penalty 10000\hskip-1\wd7\penalty 10000\box7 }

\author[Q.A. Ng\^o]{Qu\cfac oc Anh Ng\^o}
\address[Q.A. Ng\^o]{
Graduate School of Mathematical Sciences, The University of Tokyo, 3-8-1 Komaba, Meguro-ku, Tokyo 153-8914, Japan
\\
ORCID iD: 0000-0002-3550-9689}

\email{ngo@ms.u-tokyo.ac.jp, nqanh@vnu.edu.vn}

\author[Q.H. Nguyen]{Quoc-Hung Nguyen}

\address[Q.H. Nguyen]{Institute of Mathematical Sciences, ShanghaiTech University, 393 Middle Huaxia Road, Pudong, Shanghai, 201210, China}

\email{qhnguyen@shanghaitech.edu.cn}

\author[V.H. Nguyen]{Van Hoang Nguyen}

\address[V.H. Nguyen]{Department of Mathematics, FPT University, Hanoi, Vietnam}

\email{vanhoang0610@yahoo.com, hoangnv47@fe.edu.vn}

\allowdisplaybreaks

\begin{document}

\dedicatory{Dedicated to Professor Bidaut-V\'eron and Professor V\'eron on the occasion of their 70th birthday}

\begin{abstract}
For $n > k \geq 0$, $\lambda >0$, and $p, r>1$, we establish the following \textit{optimal} Hardy--Littlewood--Sobolev inequality
\[
\Big| \iint_{\R^n \times \R^{n-k}} \frac{f(x) g(y)}{ |x-y|^\lambda |y''|^\beta} dx dy \Big| 
\lesssim \| f \| _{L^p(\R^{n-k})} \| g\| _{L^r(\R^n)}
\]
with $y = (y', y'') \in \R^{n-k} \times \R^k$ under the two conditions
\[
\beta < 
\left\{
\begin{aligned}
& k - k/r   & & \text{if } \; 0 < \lambda \leq n-k,\\
& n - \lambda - k/r & & \text{if } \; n-k < \lambda,
\end{aligned}
\right.
\]
and
\[
\frac{n-k}n \frac 1p + \frac 1r + \frac { \beta + \lambda} n = 2 -\frac kn.
\]
Remarkably, there is no upper bound for $\lambda$, which is quite different from the case with the weight $|y|^{-\beta}$, commonly known as Stein--Weiss inequalities. We also show that the above condition for $\beta$ is \textit{sharp}. Apparently, the above inequality includes the classical Hardy--Littlewood--Sobolev inequality when $k=0$ and the HLS inequality on the upper half space $\R_+^n$ when $k=1$. In the unweighted case, namely $\beta=0$, our finding immediately leads to the sharp HLS inequality on $\R^{n-k} \times \R^n$ with the \textit{optimal} range $$0<\lambda<n-k/r,$$ which has not been observed before, even for the case $k=1$. Improvement to the Stein--Weiss inequality in the context of $\R^{n-k} \times \R^n$ is also considered. The existence of an optimal pair for this new inequality is also studied.
\end{abstract}

\date{\bf \today \, at \currenttime}



\maketitle


\section{Introduction}

In the existing literature, the classical Hardy--Littlewood--Sobolev inequality on $\R^n$, named after Hardy and Littlewood \cite{hl1928,hl1930} and Sobolev \cite{sobolev1938}, states that for any $n\geq 1$, $p,r > 1$, and $\lambda \in (0,n)$ satisfying the balance condition
\begin{equation}\label{eq-HLS-BC-whole}
1 /p + 1 /r +\lambda /n =2,
\end{equation}
there exists a sharp constant $N>0$ depending on $n$, $\lambda$, and $p$ such that
\begin{equation}\label{eq-HLS-whole}
\Big |\iint_{\R^n \times \R^n} \frac{f(x)g(y)}{|x-y|^\lambda} dx dy\Big| \leq N \| f \| _{L^p(\R^n )} \| g\| _{L^r(\R^n)}
\end{equation}
for any $f\in L^p(\R^n)$ and $g\in L^r(\R^n)$. The inequality \eqref{eq-HLS-whole} is also referred to as the weak form of the classical Young inequality 
\[
\Big |\int_{\R^n} f(x) (h*g)(x) dx\Big| \lesssim \| f \| _{L^p(\R^n )} \| h\| _{L^q(\R^n)} \| g\| _{L^r(\R^n)}
\]
with $p,q,r \geq 1$ and 
\[
1/p+1/q+1/r=2,
\] 
since $|x|^{-\lambda}$ belongs to the weak space $L_w^{n/\lambda} (\R^n)$; see \cite[Chapter 4]{LiebLoss}. Here and in what follows, by $\lesssim$ and $\gtrsim$ we mean inequalities up to universal constants such as $n$, $\lambda$, $p$, $r$, etc. 

Although the rough form of \eqref{eq-HLS-whole} 
was proved rather earlier, it took quite a long time to find the their sharp form until a seminal work of Lieb in 1983; see \cite{l1983}. Among other things, Lieb proved the existence of the optimal functions to the inequality \eqref{eq-HLS-whole} 
and compute the sharp constant $N$ in several special cases.

In the last two decades, the sharp HLS inequality \eqref{eq-HLS-whole} 
has captured the attention of many mathematicians and many remarkable results have already been drawn. For example, there are new methods to prove the inequality \eqref{eq-HLS-whole} and new arguments to prove the existence of the optimal functions; see \cite{Lions, CarlenLoss92, LiebLoss, fl2010, ccl2010, fl2011, fl2012b, DouQuoZhu}. In addition, one has the sharp HLS inequalities on the upper half space $\R_+^n = \R^{n-1} \times (0, +\infty)$ in \cite{DouZhu-upper, Dou2016, Gluck2020}, on bounded domains in \cite{GZ2019}, on the Heisenberg group in \cite{fl2012a, HLZ12}, and on compact Riemannian manifolds in \cite{hanzhu}. The interaction between the HLS inequality and other important inequalities has also been exploited; see \cite{beckner1993, DJ14, JN14}.

In this work, we look for a possible weighted HLS inequality on $\R^{n-k} \times \R^n$ of the following form
\begin{equation}\label{eq-wHLS-mixed}
\Big| \iint_{\R^n \times \R^{n-k}} \frac{f(x) g(y)}{ |x-y|^\lambda |y''|^\beta} dx dy \Big| 
\lesssim \| f \| _{L^p(\R^{n-k})} \| g\| _{L^r(\R^n)},
\end{equation}
where $k$ is a non-negative integer less than $n$, $x \in \R^{n-k}$, $y = (y',y'') \in \R^{n-k} \times \R^k$, and the ``distance'' $|x-y|$ is being understood as follows
\[
|x-y| = \sqrt{|x-y'|^2 + |y''|^2}.
\]
There is a number of reasons supporting us to work on the weighted HLS inequality \eqref{eq-wHLS-mixed}. For clarity, let us just mention a few connection between \eqref{eq-wHLS-mixed} and some known results, while a detailed discussion and interesting consequences will be exploited in subsection \ref{sec-Discussion} below. Clearly, the inequality \eqref{eq-wHLS-mixed} with $ \beta= 0$, if true, becomes \eqref{eq-HLS-whole} if $k=0$. In the case $k=1$, if we let $g$ be such that $g \equiv 0$ on the lower half space $\R_-^n = \R^{n-1} \times (-\infty, 0)$, then \eqref{eq-HLS-whole} becomes
\begin{equation}\label{eq-HLS-upper}
\Big |\iint_{\R_+^n \times \R^{n-1}} \frac{f(x)g(y)}{|x-y|^\lambda} dx dy\Big| \lesssim \| f \| _{L^p(\R^{n-1})} \| g\| _{L^r(\R_+^n)}.
\end{equation}
Inequality \eqref{eq-HLS-upper} is known that the HLS inequality on the upper half space $\R_+^n$ first proved by Dou and Zhu in \cite{dz2014} under the balance condition
\begin{equation}\label{eq-Identity-upper}
(n-1)/(n p) +1/r + \lambda /n =2 - 1/n.
\end{equation}
Relaxing the condition $\beta = 0$ gives
\begin{equation}\label{eq-HLSe-upper}
\Big |\iint_{\R_+^n \times \R^{n-1}} \frac{f(x)g(y)}{|x-y|^\lambda y_n^\beta} dx dy\Big| \lesssim \| f \| _{L^p(\R^{n-1})} \| g\| _{L^r(\R_+^n)},
\end{equation}
which was proved by Gluck in \cite{Gluck2020} and Liu in \cite{Liu2020} under the technical condition $\beta \leq 0$ and the balance condition
\begin{equation}\label{eq-BCe-upper}
(n-1)/(n p) +1/r  + (\lambda + \beta)/ n =2 - 1/n.
\end{equation}
On one hand, the restriction $\beta \leq 0$ in \eqref{eq-HLSe-upper} seems to be not natural from the validity of the inequality. This indicates that \eqref{eq-HLSe-upper} could be true for certain $\beta>0$. On the other hand, since the balance conditions \eqref{eq-HLS-BC-whole}, \eqref{eq-Identity-upper}, and \eqref{eq-BCe-upper} have a similar form, it is natural to ask whether or not there is a unification of \eqref{eq-HLS-whole}, \eqref{eq-HLS-upper}, and \eqref{eq-HLSe-upper}. In this paper, we aim to address these points and we are successful leading to the inequality \eqref{eq-wHLS-mixed} above. 

That said, in this work, we aim to study \eqref{eq-wHLS-mixed}. Toward a complete picture of \eqref{eq-wHLS-mixed}, our first step concerns to the validity of \eqref{eq-wHLS-mixed}. We summarize this step as the following theorem.

\begin{theorem}[weighted HLS inequality on $\R^{n-k} \times \R^n$]\label{thm-MAIN}
Let $n \geq 1$, $0\leq k <n$, $\lambda >0$, $p, r>1$, and
\begin{equation}\label{eq-Beta}
\beta < 
\left\{
\begin{aligned}
& k -\frac  kr   & & \text{if } \; 0 < \lambda \leq n-k,\\
& n - \lambda - \frac kr & & \text{if } \; n-k < \lambda  ,\\
\end{aligned}
\right.
\end{equation}
satisfying the balance condition
\begin{equation}\label{eq-Identity}
\frac{n-k}n \frac 1p + \frac 1r + \frac { \beta + \lambda} n = 2 -\frac kn.
\end{equation}
Then there exists a sharp constant $\N > 0$ such that
\begin{subequations}\label{eq-MAIN}
\begin{align}
\Big| \iint_{\R^n \times \R^{n-k}} \frac{f(x) g(y)}{ |x-y|^\lambda |y''|^\beta } dx dy \Big| 
\leq \N \| f \| _{L^p(\R^{n-k})} \| g\| _{L^r(\R^n)}
\tag*{ \eqref{eq-MAIN}$_{k,\beta}$}
\end{align}
\end{subequations}
for any functions $f\in L^p(\R^{n-k})$ and $g\in L^r(\R^n)$. Moreover, the two conditions $0<\lambda < n$ and \eqref{eq-Beta} are sharp. 
\end{theorem}

Before moving on, Theorem \ref{thm-MAIN} deserves some comments. First, it is important to note that there is no upper bound for $\lambda$, namely \eqref{eq-MAIN}$_{k,\beta}$ holds for all $\lambda > 0$ so long as $\beta$ is suitably small. Next we list a few further comments.

\begin{itemize}
  \item When $k=0$, the weight $|y''|^{-\beta}$ does not appear, hence the inequality \eqref{eq-MAIN}$_{0,\beta}$ becomes the classical HLS inequality \eqref{eq-HLS-whole} on $\R^n$. 
  
  \item When $k=1$, the inequality \eqref{eq-MAIN}$_{1,\beta}$ is essentially the same as that on $\R_+^n$ except the fact that the domain of the double integration is no longer $\R_+^n$ but the whole space $\R^n$. As a matter of fact, \eqref{eq-MAIN}$_{1,\beta}$ deals with a larger class of functions. But the sharp constant ${\mathsf N}_{n,\lambda,p}^{1,0}$ and that of \eqref{eq-HLS-upper} are related; see subsection \ref{subsec-application-extended} below.
  
  \item In all existing works on $\R_+^n$ in the literature, the condition $0< \lambda< n-1$ is always assumed. But our optimal inequality \eqref{eq-MAIN}$_{k,\beta}$ shows that this is not necessary. If we let $\beta = 0$ and $k=1$, then \eqref{eq-MAIN}$_{1,0}$ holds for $0<\lambda<n-1/r$ but does not if $n-1/r \leq \lambda < n$; see subsection \ref{subsec-application-optimal} below.
  
  \item An interesting consequence of \eqref{eq-MAIN}$_{k,\beta}$ is that it holds for any $\lambda >0$, not just $0<\lambda<n$, so long as $\beta < n-\lambda-k/r$; see subsection \ref{subsec-application-k->1} below.
    
  \item Although it is not explicitly stated in Theorem \ref{thm-MAIN}, there is a lower bound for $\beta$ because it can be easily seen from \eqref{eq-Identity}. To be more precise, we must have $\beta > -\lambda$, thanks to $p, r>1$. 
  
  \item Our inequality \eqref{eq-MAIN}$_{1,\beta}$ remains valid if we replace $|y''|^{-\beta}$ by $|y|^{-\beta}$. But in this scenario, there are some minor changes including the condition for $\lambda$; see subsection \ref{subsec-application-wHLS->SW} below. We leave this for future research.
\end{itemize}

More interesting applications of \eqref{eq-MAIN}$_{k,\beta}$ will be discussed in section \ref{sec-Discussion}. From now on, we call \eqref{eq-MAIN}$_{k,\beta}$ the \textit{optimal} HLS inequality to highlight the fact that all parameters for \eqref{eq-MAIN}$_{k,\beta}$ are in the optimal range.

As routine, the proof Theorem \ref{thm-MAIN} is carried through two steps; see section \ref{sec-RoughInequality} below. In the first step, we prove a rough form of \eqref{eq-MAIN}$_{k,\beta}$, namely without the sharp constant $\N$. Then the existence of the sharp constant $\N $ is guaranteed through the following variational problem
\begin{equation}\label{eq-VariationalProb}
\N :=\sup_{f \geq 0, g \geq 0} \big\{ \F_{\lambda,k}^{ \beta} (f,g) : \| f\| _{L^p(\R^{n-k})} =1, \| g \| _{L^p(\R^n)} = 1 \big\},
\end{equation}
where
\[
\F_{\lambda,k}^{ \beta} (f,g) = \iint_{\R^n \times \R^{n-k}} \frac{f(x) g(y)}{ |x-y|^\lambda |y''|^\beta} dx dy.
\]
As we shall soon see, in the present work, we present two different proofs for the rough inequity. These new proofs do not make use the layer cake representation technique nor the Marcinkiewicz interpolation inequality. Instead, we borrow some ideas from harmonic analysis and the theory of maximal functions.

Once we establish Theorem \ref{thm-MAIN}, it is natural to ask whether an \textit{optimal} pair $(f^\sharp, g^\sharp)$ for the weighted HLS inequality \eqref{eq-MAIN}$_{k,\beta}$, which consists of non-negative, non-trivial functions, actually exists, namely
\[
\iint_{\R^n \times \R^{n-k}} \frac{f^\sharp (x) g^\sharp (y)}{ |x-y|^\lambda |y''|^\beta } dx dy 
= \N \| f^\sharp\| _{L^p(\R^{n-k})} \| g^\sharp \| _{L^r(\R^n)}.
\]
To this purpose, let us first formally introduce an ``extension'' operator $\EE $, which turns a function $f$ on $\R^{n-k}$ to a function on $\R^n$ via the following rule
\begin{align*}
\EE [f] (y)=\int_{\R^{n-k}}\frac{f(x)dx}{ |x-y|^\lambda |y''|^\beta} \quad \text{ a.e. } \;y \in \R^n.
\end{align*}
Using this operator, we may rewrite \eqref{eq-MAIN}$_{k,\beta}$ as
\[
\Big| \int_{\R^n} (\EE [f]) (y) g(y) dy \Big| \leq \N \| f \| _{L^p(\R^{n-k})} \| g\| _{L^r(\R^n)}.
\]
Then, by duality, the HLS inequality \eqref{eq-MAIN}$_{k,\beta}$ is \textit{equivalent} to the following inequality
\begin{equation}\label{eq-MAIN-E}
\Big\|\int_{\R^{n-k}}\frac{f(x)dx}{ |x- \cdot |^\lambda |\cdot ''|^\beta} \Big\| _{L^q(\R^n)} \leq \N \| f \| _{L^p(\R^{n-k})}
\end{equation}
for any function $f\in L^p(\R^{n-k})$ with $q $ being the number
\begin{equation}\label{eq-CONDITION-E}
\frac1q = 1 - \frac 1r = \frac{n-k}n \Big( \frac1p - \frac{n-k-\lambda - \beta} {n-k} \Big) .
\end{equation}
It is important to note that $q>p$, see \eqref{eq-q>p} below. 

Similarly, one can consider the ``restriction" operator $\RR$, which maps a function $g$ on $\R^n$ to a function on $\R^{n-k}$ via the following rule
\[
\RR [g] (x) = \int_{\R^n} \frac{ g(y) dy }{|x-y|^\lambda |y''|^\beta} \quad \text{ a.e. } \;x \in \R^{n-k}.
\]
Clearly, the two operators $\EE$ and $\RR$ are dual in the sense that for any functions $f$ on $\R^{n-k}$ and $g$ on $\R^n$, the following identity
\[
\int_{\R^n} (\EE [f])(y) g(y) dy = \int_{\R^{n-k}} f(x) (\RR [g])(x) dx
\]
holds, thanks to Tonelli's theorem. Once we introduce $\RR$, we can easily see that the weighted HLS inequality \eqref{eq-MAIN}$_{k,\beta}$ is also \textit{equivalent} to the following inequality
\begin{equation*}\label{eq-MAIN-R}
\Big\| \int_{\R^n} \frac{ g(y ) dy}{|\cdot - y|^\lambda |y''|^\beta} \Big\|_{L^q (\R^{n-k})} 
\leq \N \|g\|_{L^r(\R^n)}
\end{equation*}
for any function $g\in L^r( \R^n)$ with $q>1$ satisfies
\begin{equation*}\label{eq-CONDITION-R}
\frac1q = 1 - \frac 1p = \frac n{n-k} \Big( \frac1r - \frac{n-\lambda-\beta} {n} \Big) .
\end{equation*}

Now we turn our attention to the existence of optimal pairs $(f^\sharp,g^\sharp)$ for the variational problem \eqref{eq-VariationalProb}. In view of \eqref{eq-MAIN-E}, to study the existence of optimal pairs for \eqref{eq-VariationalProb}, we study the following maximizing problem
\begin{equation}\label{eq-VariationalProb-E}
\N :=\sup_{f \geq 0} \big\{ \big\| \EE [f] \big\| _{L^q(\R^n)} : \| f\| _{L^p(\R^{n-k})} =1 \big\}.
\end{equation}
Clearly, the two maximizing problems \eqref{eq-VariationalProb} and \eqref{eq-VariationalProb-E} are actually equivalent; see section \ref{subsec-Equivalence} below. In the next result, we prove that the maximizing problem \eqref{eq-VariationalProb-E} always admits a solution $f^\sharp \in L^p (\R^{n-k})$, thus giving a solution $(f^\sharp, (\EE [f^\sharp])^{q-1})$ to the maximizing problem \eqref{eq-VariationalProb}.

\begin{theorem}[existence of optimal functions for \eqref{eq-VariationalProb-E}]\label{thm-EXISTENCE}
Suppose that all conditions in Theorem \ref{thm-MAIN} hold. Let $q$ be given by \eqref{eq-CONDITION-E}. Then, there exists a function $f^\sharp\in L^p(\R^{n-k})$ such that
\[
f^\sharp \geq 0, \quad
 \| f^\sharp\| _{L^p(\R^{n-k})} =1, \quad \text{and} \quad 
\big \| \EE [f^\sharp] \big\| _{L^q(\R^n)} = \N .
\] 
Moreover, the function $f^\sharp$ is strictly decreasing and radially symmetric with respect to some point in $\R^{n-k}$.
\end{theorem}

We prove Theorem \ref{thm-EXISTENCE} in section \ref{sec-ExistenceOptimal} below. This is done by following Talenti's proof of the sharp Sobolev inequality by considering \eqref{eq-VariationalProb-E} within the set of symmetric decreasing rearrangements. In view of the constraint in the maximizing problem \eqref{eq-VariationalProb-E}, if we denote by $f^\star$ the symmetric decreasing rearrangement with respect to $\R^{n-k}$ of a function $f \in L^p(\R^{n-k})$, then on one hand, it is well-known that 
$$\| f^\star\| _{L^p(\R^{n-k})} = \| f\| _{L^p(\R^{n-k})}$$ 
while on the other hand, there holds
\[
\big\| \EE [f] \big\| _{L^q(\R^n)} \leq \big\| \EE [f ^\star ] \big\| _{L^q(\R^n)} ;
\] 
see \eqref{eq-RieszE} below. Hence, it suffices to look for an optimal function within the set of symmetric decreasing rearrangements. 

A quick consequence of Theorem \ref{thm-EXISTENCE} is the following.

\begin{proposition}[existence of optimal pairs for \eqref{eq-VariationalProb}]\label{prop-EXISTENCE}
Assume all conditions in Theorem \ref{thm-MAIN}. Then, the sharp constant $\N $ for the inequality \eqref{eq-MAIN}$_{k,\beta}$ is achieved by some optimal pair $(f^\sharp, g^\sharp) \in L^p (\R^{n-k}) \times L^r (\R^n) $. The functions $f^\sharp$ and $g^\sharp \big|_{\R^{n-k}}$ are radially symmetric with respect to some point in $\R^{n-k}$ and monotone decreasing.
\end{proposition}

In a future work, we shall study a reverse HLS inequality on $\R^{n-k} \times \R^n$. The paper is organized as follows:
 

\tableofcontents

Before closing Introduction, let us introduce some notation and convention. For a positive integer $\ell$, we denote by $B^\ell _ R(x)$ the open ball in $\R^\ell$ centered at $x$ and radius $R$, namely
\[B^\ell _ R(x) = \{ \xi \in \R^\ell : |\xi - x| < R\}.\]
For simplicity, we often write $B^\ell _ R (0)$ as $B^\ell _ R$. 


\section{The weighted HLS inequality on $\R^{n-k} \times \R^n$}
\label{sec-RoughInequality}

In this section, we give a proof of Theorem \ref{thm-MAIN}, namely to prove \eqref{eq-MAIN}$_{k,\beta}$ without the sharp constant $\N$. 

As mentioned in Introduction, this is equivalent to showing that the supremum in \eqref{eq-VariationalProb} is finite. Toward this purpose, we first prove in subsection \ref{subsec-Equivalence} below that the two maximizing problems \eqref{eq-VariationalProb} and \eqref{eq-VariationalProb-E} are equivalent. Therefore, to prove the rough inequality \eqref{eq-MAIN}$_{k,\beta}$, it suffices to prove the rough inequality \eqref{eq-MAIN-E}, which will be done in subsection \ref{subsec-RoughInequality}. Finally, we spend subsection \ref{subsec-Necessity} to verify the necessity of the two conditions for $\lambda$ and $\beta$.


\subsection{The equivalence between (\ref{eq-VariationalProb}) and (\ref{eq-VariationalProb-E})}
\label{subsec-Equivalence}

We now prove that the two maximizing problems \eqref{eq-VariationalProb} and \eqref{eq-VariationalProb-E} are equivalent. Such a result seems to be foreseeable and standard. We provide a short proof for completeness. Denote by $N$ the supremun in \eqref{eq-VariationalProb-E}, namely
\[
N :=\sup_{f \geq 0} \big\{ \big\| \EE [f] \big\| _{L^q(\R^n)} : \| f\| _{L^p(\R^{n-k})} =1 \big\}
\] 
with $q=(1-1/r)^{-1}$. (Conventionally, we also write $(1-1/r)^{-1} = r'$.) In the first step of the proof, we show that $N = \N$. Indeed, let $(f^\sharp,g^\sharp)$ be an optimal pair for \eqref{eq-VariationalProb}. By definition, there holds 
$$\| f^\sharp\| _{L^p(\R^{n-k})} = \|g^\sharp\|_{L^r(\R^n)}=1.$$ 
As $1/q+1/r=1$, we use H\"older's inequality to get
\begin{align*}
\N 
&= \int_{\R^n} (\EE [f^\sharp])(y) g^\sharp (y) dy 
\leq \big\| \EE [f^\sharp] \big\| _{L^q(\R^n)}\| g^\sharp\| _{L^r(\R^n)} 
\leq N.
\end{align*}
Hence, we necessarily have $\N \leq N$. Now let $h^\sharp$ be an optimal function for \eqref{eq-VariationalProb-E}. Obviously, we must have $\|h^\sharp\|_{L^p(\R^{n-k})} =1$ and
\[
\| (\EE [h^\sharp])^{q-1} \| _{L^r(\R^n)}
=\Big(\int_{\R^n} (\EE [h^\sharp])^q (y) dy \Big)^{1/r} = N^{q-1},
\] 
thanks to $(q-1)r=q$. Then using \eqref{eq-MAIN}$_{k,\beta}$ applied to $(h^\sharp,(\EE [h^\sharp])^{q-1})$ we obtain
\begin{align*}
N^q &= \int_{\R^n} (\EE [h^\sharp])(y) (\EE [h^\sharp])^{q-1} (y) dy\\
&\leq \N \big\| h^\sharp \big\| _{L^p(\R^n)}\| (\EE [h^\sharp])^{q-1} \| _{L^r(\R^n)} \\
 & = \N N^{q-1}.
\end{align*}
Hence, we now get $N \leq \N$. Thus, we have just shown that $\N = N$ as claimed.

Now, we show that each optimal pair $(f^\sharp,g^\sharp)$ for \eqref{eq-VariationalProb} gives rise to an optimal function $h^\sharp$ for \eqref{eq-VariationalProb-E} and vice versa. By seeing the above calculation, this fact is quite clear. Obviously, if $(f^\sharp,g^\sharp)$ is an optimal pair for \eqref{eq-VariationalProb}, then the function $f^\sharp$ is also an optimal function for \eqref{eq-VariationalProb-E}. Conversely, if if $h^\sharp$ is an optimal function for \eqref{eq-VariationalProb-E}, then the pair $(f^\sharp, (\EE [f^\sharp])^{q-1})$ is an optimal pair for \eqref{eq-VariationalProb}.


\subsection{Proof of the weighted HLS inequality (\ref{eq-MAIN-E})}
\label{subsec-RoughInequality}

As mentioned in Introduction, to prove \eqref{eq-MAIN}$_{k,\beta}$, it suffices to prove \eqref{eq-MAIN-E}. As we shall soon see, in the present work, we present two different proofs for the rough inequity \eqref{eq-MAIN-E}. While the idea of the second proof stems from harmonic analysis and the theory of maximal functions, see Remark \ref{rmk-SecondProof} below, the idea of the first proof, which is presented in this section, demonstrates an intriguing connection between the weighted and unweighted versions of the HLS inequality; see Lemma \ref{lem-R2} below.

To begin, recall that $0< k <n$, $1<p, r<+\infty$, $\lambda >0$, and $\beta$ satisfies
\[
\beta < 
\left\{
\begin{aligned}
& k -\frac  kr   & & \text{if } \; 0 < \lambda \leq n-k,\\
& n - \lambda - \frac kr & & \text{if } \; n-k < \lambda  .\\
\end{aligned}
\right.
\] 
We do not consider the case $k=0$ since \eqref{eq-MAIN}$_{0,\beta}$ becomes the classical HLS inequality. The balance identity \eqref{eq-Identity} can be rewritten as follows
\begin{equation}\label{eq-R2}
\frac 1p + \frac 1r + \frac{\lambda - k/ q+\beta }{n-k} = 2,
\end{equation}
where $q$ is given by \eqref{eq-CONDITION-E}. Now we denote
\begin{equation}\label{eq-gamma}
\gamma = \lambda - k/ q+\beta.
\end{equation}
As $p,r>1$ we deduce from \eqref{eq-R2} that $\gamma > 0$. Now we estimate $\gamma$ from the above. As $\beta < k/q$ if $\lambda \leq n-k$, we easily obtain $\gamma < \lambda \leq n-k$ in this range. Now for $\lambda > n-k$, it follows from $\beta < n - \lambda - k/r$ that $\gamma <n-k/r-k/q=n-k$. Hence, we obtain the following important estimate
\[
0<\gamma < n-k.
\]
for all $\lambda >0$. For clarity, we split the proof into several steps. First we start with the following simple observation.

\begin{lemma}\label{lem-R1}
Let $h$ be a non-negative, non-decreasing function. Then we have
\[
\int_{0}^{+\infty}\Big( \frac{1}{\rho^{\gamma}}h(\rho)\Big) ^\tau \frac{d\rho}{\rho}
\leq 2^{(2\gamma +1) \tau }
\Big( \int_{0}^{+\infty} \frac{1}{\rho^{\gamma}}h(\rho) \frac{d\rho}\rho \Big)^\tau
\]
for any $\tau \geq 1$.
\end{lemma}

\begin{proof}
To see the inequality, we first decompose the left hand side as follows
\begin{align*}
\int_{0}^{+\infty}\Big( \frac{1}{\rho^{\gamma}}h(\rho)\Big) ^\tau \frac{d\rho}{\rho} 
&=\sum_{j \in\mathbb Z} \int_{2^j}^{2^{j+1}} \Big( \frac{1}{\rho^{\gamma}} h(\rho)\Big) ^\tau\frac{d\rho}{\rho}\\
& \leq \sum_{j \in\mathbb Z} \int_{2^j}^{2^{j+1}} \Big( \frac{1}{2^{j\gamma}} h(2^{j+1})\Big) ^\tau \frac{d\rho}{2^j} 
 = \sum_{j \in\mathbb Z} \Big( \frac{1}{2^{j\gamma}} h(2^{j+1})\Big) ^\tau,
\end{align*}
thanks to the monotonicity of $h$. Hence, by changing the index of the sum, we arrive at
\begin{align*}
\int_{0}^{+\infty}\Big( \frac{1}{\rho^{\gamma}} h(\rho)\Big) ^\tau \frac{d\rho}{\rho} 
\leq 2^{\gamma \tau} \sum_{j \in\mathbb Z} \Big( \frac{1}{2^{j\gamma}} h(2^j)\Big) ^\tau
\leq 2^{\gamma \tau} \Big( \sum_{j \in\mathbb Z} \frac{1}{2^{j\gamma}} h(2^j)\Big) ^\tau,
\end{align*}
thanks to $\tau \geq 1$. Again by the monotonicity of $h$, we see that
\begin{align*}
\frac{1}{2^{j\gamma}} h(2^j) 
= \frac 1{\log 2} 
\frac{1}{2^{j\gamma}} \int_{2^j}^{2^{j+1}} h(2^j) \frac{d\rho}\rho
\leq \frac{2^\gamma}{\log 2} \int_{2^j}^{2^{j+1}} \frac 1{\rho^{\gamma}} h(\rho) \frac{d\rho}\rho.
\end{align*}
From this we obtain
\begin{align*}
\int_{0}^{+\infty}\Big( \frac{1}{\rho^{\gamma}} h(\rho)\Big) ^\tau\frac{d\rho}{\rho} 
\leq \Big(\frac {4^{\gamma }}{\log 2} \Big)^\tau 
\Big( \int_{0}^{+\infty} \frac 1{\rho^\gamma} h(\rho) \frac{d\rho}\rho \Big)^\tau,
\end{align*}
giving the first inequality. The proof is complete.
\end{proof}

Our next step is the key step to prove \eqref{eq-MAIN-E}. The idea is to transform a weighted inequality to a suitable an unweighted inequality.

\begin{lemma}\label{lem-R2}
Let $n > k \geq 1$, $p,r>1$, and $\lambda >0$. Suppose that $\beta$ satisfies \eqref{eq-Beta}. Then for any non-negative function $h$, we have
\begin{equation}\label{eq-R1}
\begin{split}
\int_{\R^{n}}\Big( \int_{\R^{n-k}} & \frac{h(x)}{|x-y|^\lambda |y''|^\beta}dx\Big)^qdy \\
&\leq \Big(\frac{\lambda 2^{2\gamma+1}}{\gamma} \frac 1{\lambda- \gamma}\Big)^q |\mathbb S^{k-1}| 
\int_{\R^{n-k}}
\Big(\int_{\R^{n-k}}\frac{h(x)}{|z-x|^{\gamma}}dx\Big)^qdz,
\end{split}
\end{equation}
where $\gamma $ is given in \eqref{eq-gamma} and $q$ is given in \eqref{eq-CONDITION-E}.
\end{lemma}

\begin{proof}
To see this, first we notice that
\begin{equation}\label{eq-PowerPresentation}
\frac 1{|x-y|^\lambda }
= \lambda \int_{|x-y|}^{+\infty}\frac{1}{\rho^\lambda} \frac{d\rho}{\rho}
\end{equation}
with $\lambda > 0$, which, together with Fubini's theorem applied for the 90-degree cone 
$$\big\{ (x,t) : t > |x-y'| + |y''| \big\} \subset \R_+^{n-k+1},$$ 
implies
\begin{equation}\label{eq-IntegralPresentation}
\begin{aligned}
\int_{\R^{n-k}}\frac{h(x)}{|x-y|^\lambda |y''|^\beta}dx
& = \frac {\lambda}{|y''|^{\beta }} \int_{\R^{n-k}} \Big(\int_{|x-y|}^{+\infty}\frac{1}{\rho^\lambda}\frac{d\rho}{\rho} \Big) h(x) dx \\
& \leq \frac { \lambda }{|y''|^{\beta }} \int_{|y''|}^{+\infty}\frac{1}{\rho^\lambda} \Big( \int_{B^{n-k}_\rho(y')}h(x)dx \Big)\frac{d\rho}{\rho},
\end{aligned}
\end{equation}
as the above cone is contained in the ungula
\[
\big\{ (x,t) : t < |y''|, |x-y'| < t \big\} \subset \R_+^{n-k+1}.
\] 
We still need some work on \eqref{eq-IntegralPresentation}. For some $\epsilon>0$ to be determined later, we apply H\"older's inequality to get
\begin{equation}\label{eq-IntegralPresentation-AfterHolder}
\begin{aligned}
\int_{|y''|}^{+\infty}\frac{1}{\rho^\lambda} &\Big( \int_{B^{n-k}_\rho(y')}h(x)dx \Big) \frac{d\rho}{\rho} \\
&\leq \Big[ \int_{|y''|}^{+\infty}\Big( \frac{1}{\rho^{\lambda-\epsilon}}\int_{B^{n-k}_\rho(y')}h(x)dx\Big) ^q\frac{d\rho}{\rho} \Big]^{1/q}\Big[ \int_{|y''|}^{+\infty}\rho^{-\epsilon q'}\frac{d\rho}{\rho}\Big]^{1/q'}\\
& =\Big[ \frac 1{\epsilon q'} \Big]^{1/q'} \Big[ \int_{|y''|}^{+\infty}\Big( \frac{1}{\rho^{\lambda-\epsilon}}\int_{B^{n-k}_\rho(y')}h(x)dx\Big) ^q\frac{d\rho}{\rho} \Big]^{1/q}\Big[ \frac 1{|y''|^{ \epsilon q'}} \Big]^{1/q'}
\end{aligned}
\end{equation}
with $1/q + 1/q'=1$. Hence, combing \eqref{eq-IntegralPresentation} and \eqref{eq-IntegralPresentation-AfterHolder} and making use of Fubini's theorem applied for the 90-degree cone 
$$\big\{ (z,t) : t \geq |z| \big\} \subset \R_+^{k+1},$$ 
the left hand side of \eqref{eq-R1} can be estimated as follows
\begin{align*}
\int_{\R^{n}} & \Big(\int_{\R^{n-k}}   \frac{h(x)}{|x-y|^\lambda |y''|^\beta}dx\Big)^qdy \\
&\leq \lambda^q
\int_{\R^{n-k}}\int_{\R^{k}}\Big(\int_{|y''|}^{+\infty}\frac{1}{\rho^\lambda} \Big( \int_{B^{n-k}_\rho(y')}h(x)dx \Big) \frac{d\rho}{\rho}\Big)^q \frac {dy''}{|y''|^{\beta q}} dy' \\
&\leq  \lambda^q \Big[ \frac 1{\epsilon q'} \Big]^{q/q'} 
\int_{\R^{n-k}}\int_{\R^{k}} \Big[\int_{|y''|}^{+\infty}\Big( \frac{1}{\rho^{\lambda-\epsilon}}\int_{B^{n-k}_\rho(y')}h(x)dx\Big) ^q\frac{d\rho}{\rho} \Big] 
\frac{dy''}{|y''|^{\beta q+\epsilon q}} dy'\\
&= \lambda^q \Big[ \frac 1{\epsilon q'} \Big]^{q/q'} 
\int_{\R^{n-k}}
\int_{0}^{+\infty}\int_{ B^{k}_\rho(0)}\Big( \frac{1}{\rho^{\lambda-\epsilon}}\int_{B^{n-k}_\rho(y')}h(x)dx\Big) ^q \frac{dy''}{|y''|^{\beta q +\epsilon q}} \frac{d\rho}{\rho}dy' \\
& = \lambda^q \Big[ \frac 1{\epsilon q'} \Big]^{q/q'} 
 \int_{\R^{n-k}}
\int_{0}^{+\infty} \Big( \frac{1}{\rho^{\lambda-\epsilon}}\int_{B^{n-k}_\rho(y')}h(x)dx\Big) ^q \Big[ \int_{ B^{k}_\rho(0)} \frac{dy''}{|y''|^{\beta q +\epsilon q}} \Big] \frac{d\rho}{\rho}dy'.
\end{align*}
Notice that the condition $\beta < k(r-1)/r$ always holds, which yields $\beta q = \beta (1-1/r)^{-1} < k$. This together with $k \geq 1$ allows us to choose small $\epsilon > 0$ in such a way that $\beta q+\epsilon q<k$. From this we obtain
\[
\int_{ B^{k}_\rho(0)} \frac{dy''}{|y''|^{\beta q +\epsilon q}}
=\frac{|\mathbb S^{k-1}|}{k-\beta q - \epsilon q} \rho^{k-\beta q-\epsilon q},
\]
which allows us to further obtain
\begin{align*}
\int_{\R^{n}}\Big( &\int_{\R^{n-k}} \frac{h(x)}{|x-y|^\lambda |y''|^\beta}dx\Big) ^qdy\\
&\leq \frac{\lambda^q (\epsilon q')^{-q/q'} |\mathbb S^{k-1}|}{k-\beta q - \epsilon q}
\int_{\R^{n-k}}
\int_{0}^{+\infty}\Big( \frac{1}{\rho^{\lambda-\epsilon}}\int_{B^{n-k}_\rho(y')}h(x)dx\Big) ^q\rho^{k-\beta q-\epsilon q}\frac{d\rho}{\rho}dy' \\
&= \frac{\lambda^q (\epsilon q')^{-q/q'} |\mathbb S^{k-1}|}{k-\beta q - \epsilon q}
\int_{\R^{n-k}}
\int_{0}^{+\infty}\Big( \frac{1}{\rho^{\lambda-k/q+\beta}}\int_{B^{n-k}_\rho(y')}h(x)dx\Big) ^q\frac{d\rho}{\rho}dy'.
\end{align*}
Since $\epsilon>0$ is still arbitrary, we may choose one to obtain a rough constant. Observe that
\[
\min_{0 <\epsilon<k/q-\beta} \frac{ (\epsilon q')^{-q/q'} }{k-\beta q - \epsilon q} = \Big( \frac q{k-\beta q}\Big)^q
\]
at $\epsilon = (k-\beta q)/(q q')$. Hence, using this particular choice of $\epsilon$, we arrive at
\begin{equation}\label{eqKeyKey}
\begin{aligned}
\int_{\R^{n}}\Big( &\int_{\R^{n-k}} \frac{h(x)}{|x-y|^\lambda |y''|^\beta}dx\Big) ^qdy\\
\leq &\lambda^q \Big( \frac 1{k/q-\beta }\Big)^q |\mathbb S^{k-1}| 
\int_{\R^{n-k}}
\int_{0}^{+\infty}\Big( \frac{1}{\rho^{\lambda-k/q+\beta}}\int_{B^{n-k}_\rho(y')}h(x)dx\Big) ^q\frac{d\rho}{\rho}dy'.
\end{aligned}
\end{equation}
Since $h \geq 0$, the function
\[
\rho \mapsto \int_{B^{n-k}_\rho(y')}h(x)dx
\]
is non-decreasing. This together with $q\geq 1$ allows us to apply Lemma \ref{lem-R1} with $\tau$ replaced by $q$ and Fubini's theorem applied for the 90-degree cone 
\[
\big\{ (x,t) : t \geq |x-y'| \big\} \subset \R_+^{n-k+1}
\] 
to get
\begin{equation}\label{eqKeyKeyKey}
\begin{aligned}
\int_{0}^{+\infty}\Big( \frac{1}{\rho^{\lambda-k/q+\beta}} & \int_{B^{n-k}_\rho(y')}h(x)dx\Big) ^q\frac{d\rho}{\rho}\\
& \leq 2^{[2(\lambda-k/q+\beta)+1]q}
\Big( \int_{0}^{+\infty}\frac{1}{\rho^{\lambda-k/q+\beta}} \Big[ \int_{B^{n-k}_\rho(y')}h(x)dx \Big] \frac{d\rho}{\rho}\Big) ^q \\
&=2^{[2(\lambda-k/q+\beta)+1] q}
\Big(\int_{\R^{n-k} } \Big[\int_{|x-y'|}^{+\infty}\frac{1}{\rho^{\lambda-k/q+\beta}} \frac{d\rho}{\rho} \Big] h(x)dx \Big) ^q \\
&=\Big( \frac {2^{2(\lambda-k/q+\beta)+1}}{\lambda-k/q+\beta} \Big)^q \Big( \int_{\R^{n-k}}\frac{h(x)}{|x-y'|^{\lambda-k/q+\beta} }dx\Big) ^q,
\end{aligned}
\end{equation}
thanks to \eqref{eq-PowerPresentation} and $\gamma = \lambda-k/q+\beta > 0$. Putting the above estimates together, we arrive at
\begin{align*}
\int_{\R^{n}} \Big(&\int_{\R^{n-k}} \frac{h(x)}{|x-y|^\lambda |y''|^\beta}dx\Big)^qdy \\
&\leq \Big(\frac{\lambda 2^{2\gamma+1}}{\gamma} \frac 1{k/q-\beta}\Big)^q |\mathbb S^{k-1}| 
 \int_{\R^{n-k}}
\Big( \int_{\R^{n-k}}\frac{h(x)}{|x-y'|^{\gamma}}dx\Big) ^qdy'.
\end{align*} 
This is exactly the inequality \eqref{eq-R1}, and the proof of the lemma is complete.
\end{proof}

Having all the preparations above, to conclude \eqref{eq-MAIN-E}, we simply apply Lemma \ref{lem-R2} and the classical HLS inequality \eqref{eq-HLS-whole} on $\R^{n-k} \times \R^{n-k}$, namely
\begin{align*}
\big\| \EE [f] \big\| _{L^q(\R^n)} 
 \lesssim \Big[ \int_{\R^{n-k}}
\Big(\int_{\R^{n-k}}\frac{f(x)}{ |z-x|^{\gamma}}dx\Big)^qdz \Big]^{1/q} \;
\overset{\eqref{eq-HLS-whole}} \lesssim \big\| f \big\| _{L^p(\R^{n-k})},
\end{align*}
thanks to \eqref{eq-R2} and the key estimate $0<\gamma<n-k$ for all $\lambda >0$. 

Before closing this part, we prove a reverse version of \eqref{eq-R1}, see \eqref{eq-R3} below, which has its own interest. We do not directly use this result in the proof of \eqref{eq-MAIN}$_{k, \beta}$, but we shall use it in section \ref{subsec-application-optimal} to consider the HLS inequality on $\R^{n-k} \times \R^n$ with the optimal range for $\lambda$, which is quite remarkable. 

The strategy of proving \eqref{eq-R3} is similar to that of \eqref{eq-R1}, however, instead of using Lemma \ref{lem-R1}, we use a technical result from \cite{PV}, see also \cite{BVNV}, which concerns the series of equivalent norms concerning Radon measures.

\begin{lemma}\label{lem-R3}
Let $n > k \geq 1$, $p,r>1$, and $\lambda >0$. Suppose that $\beta$ satisfies \eqref{eq-Beta}. Then for any non-negative function $h$ we have
\begin{equation}\label{eq-R3}
\begin{split}
\int_{\R^{n-k}}
\Big(\int_{\R^{n-k}}\frac{h(x)}{|z-x|^{\gamma}}dx\Big)^qdz
\lesssim \int_{\R^{n}}\Big( \int_{\R^{n-k}} & \frac{h(x)}{|x-y|^\lambda |y''|^\beta}dx\Big)^qdy,
\end{split}
\end{equation}
where $\gamma $ is given in \eqref{eq-gamma} and $q$ is given in \eqref{eq-CONDITION-E}.
\end{lemma}

\begin{proof}
Our starting point is the equality in \eqref{eq-IntegralPresentation} together with Fubini's theorem applied for the 90-degree cone 
\[
\big\{ (x,t) : t \geq |x-y'| + |y''| \big\} \subset  \R_+^{n-k+1},
\]
which helps us to write
\begin{align*}
\int_{\R^{n-k}}\frac{h(x)}{|x-y|^\lambda |y''|^\beta}dx
& = \frac { \lambda }{|y''|^{\beta }} \int_{|y''|}^{+\infty}\frac{1}{\rho^\lambda} \Big( \int_{B^{n-k}_{\rho - |y''|} (y')}h(x)dx \Big)\frac{d\rho}{\rho}\\
& \geq \frac { \lambda }{|y''|^{\beta }} \int_{2|y''|}^{+\infty}\frac{1}{\rho^\lambda} \Big( \int_{B^{n-k}_{\rho - |y''|} (y')}h(x)dx \Big)\frac{d\rho}{\rho}\\
& \geq \frac { \lambda }{|y''|^{\beta }} \Big( \int_{B^{n-k}_{ |y''|} (y')}h(x)dx \Big) \Big( \int_{2|y''|}^{+\infty}\frac{1}{\rho^\lambda } \frac{d\rho}{\rho} \Big)\\
&\gtrsim \frac { \lambda }{|y''|^{\beta +\lambda }} \Big( \int_{B^{n-k}_{ |y''|} (y')}h(x)dx \Big)  .
\end{align*}
In the above estimate, the non-decreasing property of the function
\[
\rho \mapsto \int_{B^{n-k}_\rho(y')}h(x)dx
\]
have used once. Hence, we arrive at
\begin{equation}\label{eq-Lower-1}
\begin{aligned}
\int_{\R^{n}} \Big(\int_{\R^{n-k}} & \frac{h(x)}{|x-y|^\lambda |y''|^\beta}dx\Big)^qdy \\
&\geq \lambda^q
\int_{\R^{n-k}}\int_{\R^{k}}\Big(  \int_{B^{n-k}_{ |y''|} (y')}h(x)dx \Big)^q \frac {dy''}{|y''|^{(\beta + \lambda) q}} dy' .
\end{aligned}
\end{equation}
Keep in mind that $0<\lambda - k/q + \beta < n-k$. Arguing as in \eqref{eqKeyKeyKey} and making use of \cite[Proposition 5.1]{PV}, we easily get
\begin{align*}
\int_{\R^{n-k}}\int_{\R^{k}}\Big( & \int_{B^{n-k}_{ |y''|} (y')}h(x)dx \Big)^q \frac {dy''}{|y''|^{(\beta + \lambda) q}} dy'\\
&=|\mathbb S^{k-1}| \int_{\R^{n-k}} \int_0^{+\infty} \Big(\frac 1{\rho^{\lambda - k/q + \beta}} \int_{B^{n-k}_{\rho} (y')}h(x)dx \Big)^q  \frac{d\rho}{\rho} dy'\\
& \gtrsim \int_{\R^{n-k}}  \Big(  \int_0^{+\infty} \frac 1{\rho^{\lambda - k/q + \beta}} \Big[ \int_{B^{n-k}_{\rho} (y')}h(x)dx \Big] \frac{d\rho}{\rho} \Big)^q dy'\\
& =  \Big(\frac 1{\lambda - k/q + \beta} \Big)^q \int_{\R^{n-k}} \Big(  \int_{\R^{n-k}} \frac{h(x)}{|x-y'|^{\lambda - k/q + \beta}}dx  \Big)^q dy'.
\end{align*}
Combining the previous two estimates gives \eqref{eq-R3} as claimed.
\end{proof}

An immediate consequence of Lemmas \ref{lem-R2} and \ref{lem-R3} is the following, which is quite similar to \cite[estimate (2.15)]{BVHNV}.

\begin{corollary}
Let $n > k \geq 1$, $p,r>1$, and $0<\lambda <n$. Suppose that $\beta$ satisfies \eqref{eq-Beta}. Then for any non-negative function $h$ we have
\begin{align*}
\int_{\R^{n-k}}
\Big(\int_{\R^{n-k}}\frac{h(x)}{|z-x|^{\gamma}}dx\Big)^qdz
\sim \int_{\R^{n}}\Big( \int_{\R^{n-k}} & \frac{h(x)}{|x-y|^\lambda |y''|^\beta}dx\Big)^qdy,
\end{align*}
where $\gamma $ is given in \eqref{eq-gamma} and $q$ is given in \eqref{eq-CONDITION-E}.
\end{corollary}


\subsection{Necessity of the condition (\ref{eq-Beta})}
\label{subsec-Necessity}

We spend this part to discuss the necessity of the condition
\[
\beta < 
\left\{
\begin{aligned}
& k -\frac  kr   & & \text{if } \; 0 < \lambda \leq n-k,\\
& n - \lambda - \frac kr & & \text{if } \; n-k < \lambda.\\
\end{aligned}
\right.
\]
The argument performed in this part essentially follows from \cite{Ngo}.

First we establish the necessity of the condition $\beta < k(r-1)/r$ regardless of the size of $\lambda$. In this case, we may take $f \equiv \chi_{B_1^{n-k}}$ the characteristic function of $B_1^{n-k}$. Then
\begin{align*}
\big\| \EE [ \chi_{B_1^{n-k}} ] \big\| _{L^q(\R^n)} ^q 
&\geq \int_{B_1^k} \frac 1{|y''|^{\beta q}} \Big[ \iint_{(B_1^{n-k} )^2} \frac {dx dy'}{(|x-y'|^2+|y''|^2)^{\lambda/2}} \Big]^q dy'' \\&\gtrsim \int_{B_1^k} \frac {dy''}{|y''|^{\beta q}} = +\infty,
\end{align*}
as $\beta q \geq k$. Here we also use $$|x-y'|^2+|y''|^2 \leq 2(|x|^2+|y'|^2)+|y''|^2 \leq 5$$ to bound the double integral of $|x-y|^{-\lambda}$ from below. 

Notice that the above argument does not cover the range $\lambda > n-k$ since $n-\lambda - k/r < k-k/r$ in this range of $\lambda$. Hence we need extra work to cover the case $ \lambda > n-k$.

Now we rule out the case $\beta \geq n-\lambda - k/r$ for $\lambda > n-k$. For some non-negative function $f \in L^p(\R^{n-1})$ to be determined later, from \eqref{eq-Lower-1} we write 
\begin{align*}
+\infty > \| f \|_{L^p(\R^{n-1}}^q &\gtrsim \int_{\R^n} \Big(  \int_{\R^{n-k}}  \frac{f(x) }{   |x-y|^\lambda |y''|^\beta } dx \Big)^q  dy \\
&=\int_{\R^{n-k}} \int_{\R^k} \Big( \frac 1{ |y''|^\beta} \int_{\R^{n-k}}  \frac{f(x) }{   |x-y|^\lambda   } dx  \Big)^q  dy'' dy' \\
& \gtrsim
\int_{\R^{n-k}}\int_{\R^{k}}\Big( \frac 1{|y''|^{\lambda +\beta  }} \int_{B^{n-k}_{ |y''|} (y')} f(x)dx \Big)^q  dy'' dy',
\end{align*}
with $q=r/(r-1)$. Hence, by using $\int_{\R^k} = |\mathbb S^{k-1}| \int_0^{+\infty}$, we further obtain
\begin{align*}
+\infty > \| f \|_{L^p(\R^{n-1}}^q &\gtrsim\int_{\R^{n-k}}\int_0^{+\infty}\Big(  \frac 1{\rho^{ \lambda + \beta }} \int_{B^{n-k}_{\rho} (y')} f(x)dx \Big)^q \rho^{k-1} d \rho  dy'\\
&\geq  \int_{B_4^{n-k} \setminus B_2^{n-k}} \int_0^1    \Big( \frac 1{\rho^{ \lambda + \beta - k/q}   } \int_{B_{\rho}^{n-k} (y')} f(x)  dx  \Big)^q \frac{d \rho}{\rho} dy' .
\end{align*}
For the last line in the above computation, thanks to $|y'| \geq 2$ and $0\leq \rho \leq 1$, we know that the ball
\[
B_1^{n-k} \subset B_{\rho}^{n-k} (y') \subset B_5^{n-k}
\]
and that $2 \leq |y| \leq \sqrt 5$. Hence, if we choose $f = \chi_{B_6^{n-k}}$, then we can bound
\[
\int_{B_{\rho}^{n-k} (y')} f(x) dx \gtrsim \rho^{n-k},
\]
which yields
\begin{align*}
\int_0^1 \Big( \frac 1{\rho^{\lambda +\beta - k/q}}   \int_{B_{\rho}^{n-k} (y')} f(x)  dx  \Big)^q \frac{d\rho}{\rho}
& \gtrsim 
\int_0^1 \frac{d\rho}{\rho^{(\lambda + \beta - k/q + k- n)q +1} }.
\end{align*}
As $\lambda - k/q +\beta - (n-k) \geq 0$ if $\beta \geq n- \lambda - k/r$, the integral on the right hand side of the preceding inequality always diverges if $\beta \geq n - \lambda - k/r$. This completes the proof of the necessity of $\beta < n-\lambda-k/r$ in the range $ \lambda > n-k$.

Finally, we notice that the necessity of the condition $\lambda >0$. This is trivial and we can take it for granted. Otherwise, the inequality \eqref{eq-MAIN}$_{k,\beta}$ will be in the opposite direction. 


\section{Existence of an optimal pair $(f^\sharp, g^\sharp)$ for the HLS inequality}
\label{sec-ExistenceOptimal}

In this section, we prove the existence of an optimal pair $(f^\sharp, g^\sharp)$ for the optimal HLS inequality \eqref{eq-MAIN}$_{k, \beta}$ in the full regime of the parameters. Again, we do not treat the case $k=0$. Recall that $n > k \geq 1$, $p,r>1$, $\lambda >0$, and $\beta$ satisfies \eqref{eq-Beta}. In particular, there holds 
$$0<\lambda - k/q + \beta < n-k$$ for all $\lambda >0$. This together with  \eqref{eq-R2} helps us to deduce that $$1/p + 1/r>1$$ for all $\lambda >0$. Hence, we obtain
\begin{equation}\label{eq-q>p}
q := \Big( 1 - \frac 1r \Big)^{-1} >p,
\end{equation} 
which is very important in the proof. 

In this section, we prove Theorem \ref{thm-EXISTENCE} and Proposition \ref{prop-EXISTENCE}, namely there exists an optimal function $f^\sharp$ to maximizing problem \eqref{eq-VariationalProb-E}; see subsection \ref{subsec-Equivalence}. This is equivalent to proving that there exists a radially symmetric, strictly decreasing function $f^\sharp$ such that
\[
\big\| \EE [f^\sharp] \big\| _{L^q(\R^n)} = \N, \quad \| f^\sharp \| _{L^p(\R^{n-k})} =1.
\]
This is done within the first three steps of the proof. Finally, to conclude Proposition \ref{prop-EXISTENCE} and as any optimal function $f^\sharp$ for \eqref{eq-VariationalProb-E} gives rise to an optimal pair $(f^\sharp, (\EE [f^\sharp])^{q-1})$ for \eqref{eq-VariationalProb}, we shall show that the function $\EE [f^\sharp]$ is radially symmetric and strictly decreasing, which is the last step in the proof.

Throughout this section, for a function $h$, we denote by $h^\star$ the symmetric decreasing rearrangement of $h$ with respect to the first $n-k$ coordinates; see \cite{LiebLoss} or \cite{Bur} for the definition. Now we prove the existence of a non-trivial maximizer $f^\sharp$ for the problem \eqref{eq-VariationalProb-E}. For the sake of clarity, we divide our proof into several steps.

\medskip
\noindent\textbf{Step 1}. \textit{Selecting a suitable minimizing sequence for \eqref{eq-VariationalProb-E}.} 

We start our proof by letting $( f_j )_j$ be a maximizing sequence in $L^p(\R^{n-k})$ for the problem \eqref{eq-VariationalProb-E} such that $f_j$ is non-negative. Keep in mind that
\begin{align*}
\| f_j\| _{L^p(\R^{n-k})} = \| (f_j)^\star\| _{L^p(\R^{n-k})}
\end{align*}
Now by using Riesz's rearrangement inequality, see \cite[chapter 3]{LiebLoss}, H\"older's inequality, and $1/q+1/r=1$, we know that
\begin{equation}\label{eq-RieszE}
\begin{aligned}
\big\| \EE [ f_j ] \big\| _{L^q(\R^n)} 
&=
\sup_{\| h\| _{L^r (\R^n)} =1 } \int_{\R^k} \frac 1{ |y''|^{\beta }}
 \Big[ \iint_{(\R^{n-k} )^2 }\frac{f_j(x) h(y', y'') dx dy'}{ \big( |x-y'|^2 + |y''|^2 \big)^{\lambda/2} } \Big] dy'' \\
&\leq \sup_{\| h \| _{L^r (\R^n)} = 1} \int_{\R^k} \frac 1{ |y''|^{\beta }} 
 \Big[ \iint_{(\R^{n-k} )^2 }\frac{ (f_j)^\star (x) h^\star (y', y'') dx dy'}{ \big( |x-y'|^2 + |y''|^2 \big)^{\lambda/2} } \Big] dy''\\
&\leq \sup_{\| h \| _{L^r (\R^n)} = 1} \big\| \EE [ (f_j)^\star ] \big\| _{L^q(\R^n)} \| h^\star \| _{L^r (\R^n)} \\
&= \big\| \EE [ (f_j)^\star ] \big\| _{L^q(\R^n)}.
\end{aligned}
\end{equation}
Notice that
\begin{align*}
 \int_{\R^n} |h|^r dy &= \int_{\R^k} \Big( \int_{\R^{n-k}} |h|^r(y',y'') dy' \Big) dy''\\
& =\int_{\R^k} \Big( \int_{\R^{n-k}} |h^\star|^r(y',y'') dy' \Big) dy'' = \int_{\R^n} |h^\star|^r dy.
\end{align*}
Hence, as $\| h\| _{L^r (\R^n)} =1 $ we deduce that $\| h^\star \| _{L^r (\R^n)} = 1$. Thus
\begin{align*}
\big\| \EE [ f_j ] \big\| _{L^q(\R^n)} \leq\big\| \EE [ (f_j)^\star ] \big\| _{L^q(\R^n)} .
\end{align*}
Putting the above two estimates between $f_j$ and $(f_j)^\star$ together, we may further assume that $f_j$ is radially symmetric with respect to the origin and non-increasing. By abusing notations, we shall write $f_j(x)$ by $f_j(|x|)$ or even by $f_j (r)$ where $r=|x|$. We can normalize $f_j$ in such a way that $\| f_j\| _{L^p(\R^{n-k})} =1$. From this and the monotonicity of $f_j$, we have
\[
\begin{split}
1 = & |\mathbb S^{n-k-1}| \int_0^\infty f_j(r)^p r^{n-k-1} dr \geq \frac{|\mathbb S^{n-k-1}|}{n-k} f_j(R)^p R^{n-k}
\end{split}
\]
for any $R > 0$. From this, we obtain the following estimate 
\begin{equation}\label{eq-fj<=}
0\leq f_j(r) \leq \Big(\frac{n-k}{|\mathbb S^{n-k-1}|}\Big)^{1/p} r^{-(n-k)/p}
\end{equation}
for any $r > 0$. 

\medskip
\noindent\textbf{Step 2}. \textit{Existence of a potential maximizer $f^\sharp$ for the problem \eqref{eq-VariationalProb-E}.} 

For each non-negative function $h$ on $\R^{n-k}$, we denote
\[
\| h \| _* = \sup_{x \in \R^{n-k} , \rho>0} \Big[ \rho^{-\frac{n-k}{p'}} \int_{ B_\rho^{n-k} (x) } h(z) dz \Big]
\]
with $ p' = p/(p-1)$. Suppose that $h \in L^p(\R^{n-k})$. By H\"older's inequality we have
\begin{align*}
\rho^{-\frac{n-k}{p'} }\int_{B_\rho^{n-k} (x)} h(z) dz \lesssim \| h \| _{L^p(\R^{n-k})}.
\end{align*}
for arbitrary $x \in \R^{n-k}$ and for any $\rho >0$. Hence, by definition we get
\[
\| h \| _* \lesssim \| h\| _{L^p(\R^{n-k})}. 
\]
To go further, we need an auxiliary result, an analogue of \cite[Lemma 2.4]{l1983} concerning the behavior of $\big\| \EE [ f ] \big\| _{L^q(\R^n)}$, whose proof is located in Appendix \ref{apd-FarAwayFromZero} 

\begin{lemma}\label{lem-FarAwayFromZero}
Suppose that $f\in L^p(\R^{n-k})$ is non-negative. Then there exists a constant $C_1 > 0$, independent of $f$ such that 
\begin{equation*}\label{eq:farawayzero}
 \int_{\R^n} (\EE [ f ])^q dy \leq C_1 \| f\| _*^{q - p } \int_{\R^{n-k}} f^p dx 
\end{equation*}
where $q = r/(r-1) > p> 1$.
\end{lemma}

Going back to the maximizing sequence $( f_j )_j$ in $L^p(\R^{n-k})$ for the problem \eqref{eq-VariationalProb-E}, for each $j$ we set
\[
a_j = \sup_{r>0} \big[ r^\frac{n-k}p f_j (r) \big].
\]
In view of \eqref{eq-fj<=} we know that
\[
0 < a_j \leq \Big(\frac{n-k}{|\mathbb S^{n-k-1}|}\Big)^{1/p}
\]
for all $j$. Using the monotonicity of $f_j$, we deduce that
\begin{align*}
 \int_{ B_\rho^{n-k} (x) } f_j (z) dz 
& \leq a_j \int_{ B_\rho^{n-k} (x) } |z|^{-\frac{n-k}p} dz \\
& \leq a_j \int_{ B_\rho^{n-k} (0) } |z|^{-\frac{n-k}p} dz \\
&= \frac{a_j}{(n-k)(1-1/p)}\rho^{ (n-k) (1 - \frac{1}{p})}
\end{align*}
for arbitrary $x \in \R^{n-k}$ and for any $\rho >0$. Consequently, there holds
\[
\| f_j \| _* \leq \frac{a_j}{(n-k)(1-1/p)}
\]
for all $j$. Recall from the choice of $f_j$ the following
\[
\| f_j\| _{L^p(\R^{n-k})} = 1, \quad 
\big\| \EE [ f_j ] \big\| _{L^q(\R^n)} \to \N .
\] 
Making use of Lemma \ref{lem-FarAwayFromZero} above, we obtain the following estimate
\begin{align*}
( \N)^q 
& \leq \int_{\R^n} (\EE [ f_j ])^q dy + o(1)_{j \nearrow +\infty} \\
& \leq C_1 \| f_j\| _*^{q - p } \int_{\R^{n-k}} f_j^p dx + o(1)_{j \nearrow +\infty} \\
& = C_1 \| f_j\| _*^{q - p } + o(1)_{j \nearrow +\infty} .
\end{align*}
Keep in mind that $q>p$. Hence, $\| f_j\| _*$ is bounded from below away from zero. This together with $\| f_j\| _* \lesssim a_j$ allows us to assume that $a_j \geq 2c_0$ for some $c_0 >0$ and for all $j$. Consequently, for each $j$, we can choose $\lambda_j > 0$ in such a way that
\[
\lambda_j^\frac{n-k}{p} f_j(\lambda_j) > c_0.
\] 
Then we set 
\[
g_j(x) = \lambda_j^\frac{n-k} p f_j(\lambda_j x).
\] 
From this, it is routine to check that $( g_j )_j$ is also a minimizing sequence for problem \eqref{eq-VariationalProb-E}, however, $g_j(1) > c_0$ for any $j$ by our choice for $\lambda_j$. Consequently, by replacing the sequence $( f_j )_j$ by the new sequence $( g_j )_j$, if necessary, we can further assume that our sequence $( f_j )_j$ obeys
\[
f_j(1) > c_0 \quad \text{ for any }\; j.
\] 

Similar to Lieb's argument in \cite{l1983}, which is based on Helly's theorem, by passing to a subsequence, we have
\[
f_j\to f^\sharp \quad \text{a.e. in } \; \R^{n-k}.
\] 
It is now evident that $f^\sharp$ is non-negative, radially symmetric, non-increasing, and is in $L^p(\R^{n-k})$. Of course, there holds $f^\sharp \not \equiv 0$. The rest of our arguments is to show that $f^\sharp$ is indeed the desired minimizer for \eqref{eq-VariationalProb-E}.

\medskip
\noindent\textbf{Step 3}. \textit{The function $f^\sharp$ is a maximizer for \eqref{eq-VariationalProb-E}.} 

Recall that $( f_j )_j$ is a minimizing sequence for the problem \eqref{eq-VariationalProb-E} and $f_j \to f^\sharp$ a.e. in $\R^{n-k}$. The limit function $f^\sharp$ satisfies $\| f^\sharp\| _{L^p(\R^{n-k})}>0$ because $f_j(x) > c_0$ for any $j$ and all $|x| \leq 1$. To go further, we need the following auxiliary result.

\begin{lemma}\label{lem-PointwiseConvergence}
Suppose that $(f_j)_j$ is a sequence of non-negative functions satisfying 
\[
f_j (x) \leq C |x|^{-\frac{n-k}p}
\]
for all $x\in \R^{n-k}$ and for some $C>0$. Then, if $f_j\to f^\sharp$ a.e. in $\R^{n-k}$, then we have
\[
\EE [ f_j ](y) \to \EE [ f^\sharp ](y)
\]
for almost every $y \in \R^n$. 
\end{lemma}

Lemma \ref{lem-PointwiseConvergence} above simply follows from the dominated convergence theorem. It is worth noting that in order to apply the dominated convergence theorem, we make use of the inequality
$$\lambda + (n-k)/p > n-k,$$ 
which always holds true under our assumption \eqref{eq-Identity}. Hence, we omit the details and its proof is left for interested readers. 

Using Lemma \ref{lem-PointwiseConvergence} above, we further know that $\EE [ f_j ] \to \EE [ f^\sharp ] $ for a.e. in $\R^n$. The rest of the proof is more or less standard; see \cite[Lemma 2.7]{l1983}. Applying the Brezis--Lieb lemma to get
\[
\int_{\R^{n-k}} \big| |f_j|^p - |f^\sharp|^p - |f_j - f^\sharp|^p \big| dx \to 0
\]
as $j \nearrow +\infty$. So, one one hand we have
\begin{equation}\label{eq-LR}
\| f_j - f^\sharp \| _{L^p(\R^{n-k})}^p = 1 - \| f^\sharp \| _{L^p(\R^{n-k})}^p + o(1)_{j \nearrow +\infty}.
\end{equation}
However, on the other hand, we can estimate
\begin{align*}
( \N)^q + o(1)_{j \nearrow +\infty } &=\big \| \EE [ f_j ] \big\| _{L^q(\R^n)} ^q \\
 &=\big\| \EE [ f^\sharp] \big\| _{L^q(\R^n)}^q + \big\| \EE [ f_j - f^\sharp ] \big\| _{L^q(\R^n)} ^q + o(1)_{j \nearrow +\infty }\\
 &\leq (\N) ^q  \big[ \| f^\sharp\| _{L^p(\R^n)}^q +  \| f_j - f^\sharp \| _{L^p(\R^n)}^q \big] + o(1)_{j \nearrow +\infty }.
\end{align*}
Thus, dividing both sides of the preceding computation by $( \N)^q$ gives
\begin{equation}\label{eq-RL}
1 \leq \| f^\sharp\| _{L^p(\R^{n-k})}^q + \| f_j - f^\sharp\| _{L^p(\R^{n-k})}^q + o(1)_{j \nearrow +\infty}.
\end{equation}
Combining \eqref{eq-LR} and \eqref{eq-RL} and sending $j \nearrow +\infty$, we arrive at
\[
1 \leq \| f^\sharp \| _{L^p(\R^{n-k})}^q + \big(1-\| f^\sharp\| _{L^p(\R^{n-k})}^p \big)^{q/p}.
\]
From the fact that $q>p$ seeing \eqref{eq-q>p} and that $\| f^\sharp \| _{L^p(\R^{n-k})}>0$, we deduce that
\[
\| f^\sharp\| _{L^p(\R^{n-k})}= 1.
\]
This shows that $f^\sharp$ is a minimizer for \eqref{eq-VariationalProb-E}; hence finishing the proof of Step 3.

\medskip
\noindent\textbf{Step 4}. \textit{The function $\EE [f^\sharp]$ has two symmetries in $y'$ and $y''$ and is strictly decreasing in $y'$} 

This step is for the proof of Proposition \ref{prop-EXISTENCE}. We show that $\EE [f^\sharp] $ of variable $y$ has two symmetries in $y'$ and $y''$. While the symmetry with respect to $y''$ is obvious from the definition of $\EE [f^\sharp$, the symmetry with respect to $y'$ is also clear since $\EE [f^\sharp] $ is essentially the convolution of two radially symmetric functions $f^\sharp$ and $(|\cdot|^2+|y''|^2)^{-\lambda/2}$; see \cite[Lemma 2.2(i)]{l1983}. Since the argument is simple and short, we provide a proof for completeness. Indeed, let $A \in O(n-k)$ be arbitrary. Then
\begin{align*}
(\EE [f^\sharp])(Ay', y'') &= \frac 1{|y''|^\beta} \int_{\R^{n-k}} \frac{f^\sharp (x) }{(| x- A y'|^2 + |y''|^2)^{\lambda/2}} dx\\
&= \frac 1{|y''|^\beta} \int_{\R^{n-k}} \frac{f^\sharp (x) }{(|A(A^t x-y')|^2 + |y''|^2)^{\lambda/2}} dx\\
&= \frac 1{|y''|^\beta} \int_{\R^{n-k}} \frac{f^\sharp (A^t x) }{(| A^t x-y' |^2 + |y''|^2)^{\lambda/2}} dx\\
&= \frac 1{|y''|^\beta} \int_{\R^{n-k}} \frac{f^\sharp ( x) }{(| x-y' |^2 + |y''|^2)^{\lambda/2}} |\det A| dx\\
&=(\EE [f^\sharp])(y', y'') ,
\end{align*}
where $A^t$ is the transpose of $A$. Finally, the monotonicity of $\EE [f^\sharp]$ in $y'$ follows from \cite[Lemma 2.2(ii)]{l1983}. 

Notice that the monotonicity of $\EE [f^\sharp]$ in $y'$ can also be derived from a general result of Anderson applied to the function $h = (|\cdot|^2+|y''|^2)^{-\lambda/2} f^\sharp (\cdot + y')$ for $y$ fixed; see \cite[Theorem 1]{Anderson}. This is because
\begin{align*}
(\EE [f^\sharp])(\tau y', y'') 
&\geq \int_{\R^{n-k}} \frac{f^\sharp ( x + (1-\tau)y') }{(| x- \tau y' |^2 + |y''|^2)^{\lambda/2}} dx\\
&=\int_{\R^{n-k}} h(x - \tau y')dx\\
& \geq \int_{\R^{n-k}} h(x -y')dx \\
&=(\EE [f^\sharp])( y', y'')
\end{align*}
for any $0 \leq \tau \leq 1$. Here the radial symmetry and monotonicity of $f^\sharp$ are crucial to get
\[
f^\sharp ( x) \geq f^\sharp ( x + (1-\tau)y').
\]
If $\beta \geq0$, then the monotonicity of $\EE [f^\sharp]$ in $y''$ is clear. But it is not clear if this still holds when $\beta < 0$.

Before closing this section, we have the following remark.

\begin{remark}\label{rmk-SecondProof}
As $\| f_*\| \lesssim \| f\| _{L^p(\R^{n-k})}$, Lemma \ref{lem-FarAwayFromZero} gives us another proof of the rough HLS inequality \eqref{eq-MAIN}$_{k,\beta}$.
\end{remark}


\section{Discussions}
\label{sec-Discussion}

This section is devoted to a number of discussion and application from simple to complex around our main inequality \eqref{eq-MAIN}$_{k, \beta}$. 


\subsection{The HLS inequality on $\R^{n-k} \times \R^n$ with optimal range $0<\lambda<n-k/r$}
\label{subsec-application-optimal}

We start this section with a quite surprise application of Theorem \ref{thm-MAIN}. To be more precise, with $\beta=0$, which is possible because $\lambda < n - k/r$, we obtain from Theorem \ref{thm-MAIN} the following optimal result.

\begin{theorem}[optimal HLS inequality on $\R^{n-k} \times \R^n$]\label{thm-MAIN-unweighted}
Let $n \geq 1$, $0\leq k <n$, $p, r>1$, and $\lambda \in (0,n-k/r)$ satisfying the balance condition
\begin{equation*}\label{eq-Identity-unweighted}
\frac{n-k}n \frac 1p + \frac 1r + \frac { \lambda} n = 2 -\frac kn.
\end{equation*}
Then there exists a sharp constant ${\mathsf N}_{n,\lambda,p}^{k,0}  > 0$ such that
\begin{equation}\label{eq-MAIN-unweighted}
\Big| \iint_{\R^n \times \R^{n-k}} \frac{f(x) g(y)}{ |x-y|^\lambda  } dx dy \Big| 
\leq {\mathsf N}_{n,\lambda,p}^{k,0}
 \| f \| _{L^p(\R^{n-k})} \| g\| _{L^r(\R^n)}
\end{equation}
for any functions $f\in L^p(\R^{n-k})$ and $g\in L^r(\R^n)$. 
\end{theorem}

For arbitrary $r>1$, Theorem \ref{thm-MAIN-unweighted} is optimal in the sense that it does not hold if $\lambda \geq n-k/r$ by seeing \eqref{eq-Beta}. However, if we fix $0<\lambda<n$, then resolving the inequality $\lambda < n-k/r$ gives
\[
r > \max \Big\{ 1, \frac k{n-\lambda} \Big\}.
\]  
This condition tells us that the closer to $n$ the parameter $\lambda$ is, the bigger $r$ is. In the special case $k=1$, Theorem \ref{thm-MAIN-unweighted} helps us to revisit the HLS inequality \eqref{eq-HLS-upper} on the upper half space $\R_+^n$ with the \textit{optimal} range $$0<\lambda<n-1/r.$$ This improves the result of Dou and Zhu in \cite{dz2014}, which is stated for $0<\lambda<n-1$.
 

\subsection{An improvement of the HLS inequality on $\R_+^n$ with extended kernel (\ref{eq-HLSe-upper}) in the regime $0<\lambda < n-1$ and $0<\beta<1 - 1/r$}
\label{subsec-application-extended}

Our motivation of working on this problem also comes from the fact that the HLS inequality \eqref{eq-HLSe-upper} on $\R_+^n$ with extended kernel is ``weaker'' than the HLS inequality \eqref{eq-HLS-upper} on $\R_+^n$. Here by ``weaker'' we mean we can use \eqref{eq-HLS-upper} to derive \eqref{eq-HLSe-upper}. Indeed, as $\beta \leq 0$, we clearly have 
$$y_n^{-\beta} \leq |x-y|^{- \beta},$$ 
giving
\[
\iint_{\R_+^n \times \R^{n-1}} \frac{f(x)g(y)}{|x-y|^\lambda y_n^\beta} dx dy
\leq \iint_{\R_+^n \times \R^{n-1}} \frac{f(x)g(y)}{|x-y|^{\lambda +\beta}} dx dy,
\]
where, for simplicity, all the functions $f$ and $g$ are being non-negative. Notice that the balance condition \eqref{eq-BCe-upper} allows us to apply \eqref{eq-HLS-upper} with $\lambda$ replaced by $\lambda + \beta$. To be more precise, fixing any $f \in L^p(\partial\R_+^n)$ and any $g \in L^r(\R_+^n)$, we have
\begin{align*}
\iint_{\R_+^n \times \R^{n-1}} \frac{f(x)g(y)}{|x-y|^\lambda y_n^\beta} dx dy
&\leq \iint_{\R_+^n \times \R^{n-1}} \frac{f(x)g(y)}{|x-y|^{\lambda +\beta}} dx dy \\
& \overset{\eqref{eq-HLS-upper}}  \lesssim \| f\| _{L^p(\partial\R_+^n)}\| g\| _{L^r(\R_+^n)},
\end{align*}
provided \eqref{eq-BCe-upper} and $\beta \leq 0$ hold. This explains why \eqref{eq-HLSe-upper} is weaker than \eqref{eq-HLS-upper}. 

From the above discussion, it is natural to ask if \eqref{eq-HLSe-upper} still holds for suitable $\beta >0$. If there is such an inequality, then it implies that we will have a ``stronger'' version of the HLS inequality \eqref{eq-HLS-upper} on $\R_+^n$ in the following sense
\begin{align*}
\iint_{\R_+^n \times \R^{n-1}} \frac{f(x)g(y)}{|x-y|^{\lambda}} dx dy
& \leq \iint_{\R_+^n \times \R^{n-1}} \frac{f(x)g(y)}{|x-y|^{\lambda - \beta} y_n^\beta} dx dy \\
&\overset{?}{\lesssim} \| f\| _{L^p(\partial\R_+^n)}\| g\| _{L^r(\R_+^n)},
\end{align*}
so long as \eqref{eq-HLS-upper} holds. 

To be able to compare, we limit ourselves to the case $0<\lambda<n-1$, hence $\beta < 1-1/r$. Clearly, under the above setting, our inequality \eqref{eq-MAIN}$_{1, \beta}$ becomes
\begin{align*}
\Big| \iint_{\R^n \times \R^{n-1}} \frac{f(x) g(y)}{ |x-y|^\lambda |y_n|^\beta } dx dy \Big| 
\leq \mathsf N_{n, \lambda, p}^{1, \beta} \| f \| _{L^p(\R^{n-1})} \| g\| _{L^r(\R^n)}
\end{align*}
with $0<\lambda<n-1$ and $\beta < 1-1/r$; hence providing us an improvement of \eqref{eq-HLS-upper} for $\beta$, if we let $g |_{\R_-^n} \equiv 0$. 

In subsection \ref{subsec-application-k->1} below, we show that there is another way, which is quite intriguing, to obtain \eqref{eq-HLSe-upper} directly from \eqref{eq-MAIN}$_{k, \beta}$ without assuming $k=1$. Moreover, it is quite interesting to note that from the argument leading to \eqref{eq-MAIN}$_{k, \beta}$, we can relate the sharp constant $\N$ and that of \eqref{eq-HLS-upper}.


\subsection{From $\R^{n-k} \times \R^n$ to $\R^{n-k} \times \R_+^{n-k+1}$}
\label{subsec-application-k->1}

An other idea to improve \eqref{eq-HLSe-upper} for possible $\beta>0$ is to transform \eqref{eq-MAIN}$_{k, \beta}$ into \eqref{eq-HLSe-upper}. To fix the idea and for simplicity, we still limit ourselves to the case $0<\lambda<n-k$, hence $0<\beta < k(1-1/r)$.

To transform \eqref{eq-MAIN}$_{k, \beta}$ into \eqref{eq-HLSe-upper}, we simply make use of \eqref{eq-MAIN}$_{k, \beta}$ for function $g$ being radially symmetric in the last $k$ coordinates, namely
\[
g (y',y'') = g (y', |y''|).
\]
Observe that
\[
\int_{\R^n} \big(g (y)\big)^r dy=|\mathbb S^{k-1}| \int_{\R^{n-k}} \int_0^{+\infty} \big(g (y',\rho)\big)^r \rho^{k-1}d\rho dy' .
\]
Hence, by setting
\[
G (y', |y''|) =|\mathbb S^{k-1}|^{1/r} g (y', y'') |y''|^{(k-1)/r},
\]
on one hand, it is easy to verify that
\[
\int_{\R^n} g^r (y) dy=\int_{\R_+^{n-k+1}} G^r (z, \rho) dz d\rho.
\]
On the other hand, by using $\int_{\R^k} = |\mathbb S^{k-1}| \int_0^{+\infty}$, we get
\begin{align*}
\F_{\lambda,k}^{0, \beta} (f, g ) &=
 |\mathbb S^{k-1}| \int_{\R^{n-k}} \int_{\R^{n-k}}\int_0^{+\infty} \frac{f (x) g (y', \rho)}{ \big[ |x-y'|^2+\rho^2 \big]^{\lambda/2} \rho^\beta} \rho^{k-1}d\rho dy' dx \\
&= |\mathbb S^{k-1}|^{1-1/r} \int_{\R^{n-k}} \int_{\R_+^{n-k+1}} \frac{f (x) G (y', \rho) } { \big[ |x-y'|^2+\rho^2 \big]^{\lambda/2} \rho^{\beta - (k-1)(1 -1/r)}} d\rho dy' dx .
\end{align*}
Hence, the inequality \eqref{eq-MAIN}$_{k,\beta}$ for $(f , g )$ becomes the following inequality for $(f, G)$
\begin{equation}\label{eq-MAIN-transformed}
\begin{aligned}
\int_{\R^{n-k}} \int_{\R_+^{n-k+1}} \frac{f (x) G (y', \rho) }{ \big[ |x-y'|^2+\rho^2 \big]^{\lambda/2} \rho^{\widehat\beta} } & d\rho dy' dx \\
\leq |\mathbb S^{k-1}|^{1/r -1 } \N & \| f \| _{L^p(\R^{n-k})} \| G \| _{L^r (\R_+^{n-k+1})} 
\end{aligned}
\end{equation}
with $$\widehat \beta = \beta - (k-1)\big(1 -1/r \big) < 1 - 1/r.$$ 
Obviously, if $(f^\sharp, g^\sharp)$ is an optimal pair for \eqref{eq-MAIN}$_{k,\beta}$, then $(f^\sharp, G^\sharp)$ is also an optimal pair for \eqref{eq-MAIN-transformed}. However, due to the above transformation, the regularity for $(f^\sharp, g^\sharp)$ and $(f^\sharp, G^\sharp)$ are quite different. This could shed some light on the problem of classification of all optimal pairs for the inequality on the upper half space.

Nevertheless, from the above derivation and the existence of an optimal pair $(f^\sharp, g^\sharp)$ for \eqref{eq-MAIN}, we obtain the following improvement of \eqref{eq-HLSe-upper}, which is optimal.

\begin{theorem}[]\label{thm-MAIN-upper}
Let $n \geq 2$, $\lambda \in (0,n-1)$, $\beta < (r-1)/r$, and $p, r>1$ satisfying the balance condition
\begin{equation}\label{eq-Identity-upper}
\frac{n-1}n \frac 1p + \frac 1r + \frac { \beta + \lambda} n = 2 -\frac 1n.
\end{equation}
Then there exists a sharp constant $\P > 0$ such that
\begin{subequations}\label{eq-MAIN-upper}
\begin{align}
\Big| \iint_{\R_+^n \times \R^{n-1}} \frac{f(x) g(y)}{ |x-y|^\lambda y_n^\beta } dx dy \Big| 
\leq \P \| f \| _{L^p(\R^{n-1})} \| g\| _{L^r(\R_+^n)}
\tag*{ \eqref{eq-MAIN-upper}$_{\beta}$}
\end{align}
\end{subequations}
for any functions $f\in L^p(\R^{n-1})$ and $g\in L^r( \R_+^n)$. 
\end{theorem}

In the last part of this subsection, we show that the transformed inequality \eqref{eq-MAIN-transformed} reveals an connection between the two sharp constants $\P$ and $\N$. We turn this observation into a proposition as follows.

\begin{proposition}[]\label{prop-COMPARISION}
There holds
\[
\mathsf P_{n-k+1, \lambda, p}^{ \beta} = |\mathbb S^{k-1}|^{1/r -1 } \mathsf N_{n,\lambda,p}^{k, \beta+(k-1)(1-1/r)}.
\]
In particular, with $k=1$ we have
\[
\P = 2^{1/r -1 } \mathsf N_{n,\lambda,p}^{1, \beta},
\]
relating the sharp constant of the HLS inequalities on $\R^{n-1} \times \R_+^n$ and on $\R^{n-1} \times \R^n$.
\end{proposition}

\begin{proof}
We apply \eqref{eq-MAIN-transformed} for $(f^\sharp, G^\sharp)$ being an optimal pair for \eqref{eq-MAIN-transformed} to get
\begin{align*}
\mathsf P_{n-k+1, \lambda, p}^{\widehat \beta} &\| f^\sharp \| _{L^p(\R^{n-k})} \| G^\sharp \| _{L^r (\R_+^{n-k+1})} \\
&\geq \int_{\R^{n-k}} \int_{\R_+^{n-k+1}} \frac{f^\sharp (x) G^\sharp (y', \rho) }{ \big[ |x-y'|^2+\rho^2 \big]^{\lambda/2} \rho^{\widehat\beta} } d\rho dy' dx \\
&= |\mathbb S^{k-1}|^{1/r -1 } \N \| f^\sharp \| _{L^p(\R^{n-k})} \| G^\sharp \| _{L^r (\R_+^{n-k+1})} .
\end{align*}
From this we obtain
\[
\mathsf P_{n-k+1, \lambda, p}^{\widehat \beta} \geq |\mathbb S^{k-1}|^{1/r -1 } \N.
\]
Now we use \eqref{eq-MAIN-upper}$_{\widehat \beta}$ with $(f^\sharp, g^\sharp)$ being an optimal pair for \eqref{eq-MAIN-upper}$_{\widehat \beta}$ to get
\begin{align*}
\mathsf P_{n-k+1, \lambda, p}^{\widehat \beta} & \| f^\sharp \| _{L^p(\R^{n-k})} \| g^\sharp \| _{L^r (\R_+^{n-k+1})}\\
&= \int_{\R^{n-k}} \int_{\R_+^{n-k+1}} \frac{f^\sharp (x) g^\sharp (y', \rho) }{ \big[ |x-y'|^2+\rho^2 \big]^{\lambda/2} \rho^{\widehat\beta} } d\rho dy' dx \\
&\leq |\mathbb S^{k-1}|^{1/r -1 } \N \| f^\sharp \| _{L^p(\R^{n-k})} \| g^\sharp \| _{L^r (\R_+^{n-k+1})} .
\end{align*}
Hence we get
\[
\mathsf P_{n-k+1, \lambda, p}^{\widehat \beta} \leq |\mathbb S^{k-1}|^{1/r -1 } \N.
\]
The proof follows by putting the above estimates together.
\end{proof}


\subsection{An Stein--Weiss type inequality on $\R^{n-k} \times \R^n$}
\label{subsec-application-wHLS->SW}

In the literature, another weighted version of the inequality \eqref{eq-HLS-whole}, or the doubly weighted HSL inequality, also known as the SW inequality, named after Stein and Weiss, was also obtained in \cite{st1958}. Roughly speaking, for $$0<\lambda<n$$ and for suitable $\alpha, \beta$ satisfying
\[
\alpha < n(p-1)/p, \quad \beta < n(r-1)/r, \quad \alpha + \beta \geq 0,
\] 
the following rough inequality holds
\begin{equation}\label{eq-SW-whole}
\Big |\iint_{\R^n \times \R^n} \frac{f(x)g(y)}{|x|^\alpha |x-y|^\lambda |y|^\beta} dx dy\Big| \lesssim \| f \| _{L^p(\R^n )} \| g\| _{L^r(\R^n)}
\end{equation}
for any $f\in L^p(\R^n)$ and any $g\in L^r(\R^n)$ together with 
\[
1/p + 1/r \geq 1
\]
and the new balance condition
\begin{equation*}\label{eq-SW-BC-whole}
1 / p + 1/ r + (\alpha+\beta+\lambda )/n =2.
\end{equation*}
Inequality \eqref{eq-SW-whole} was extended by Dou to the case of $\R_+^n$; see \cite{Dou2016}. To be more precise, under the condition 
\[
0<\lambda<n-1
\]
the following sharp inequality was proved
\begin{equation}\label{eq-SW-upper}
\Big |\iint_{\R_+^n \times \R^{n-1}} \frac{f(x)g(y)}{|x|^\alpha |x-y|^\lambda |y|^\beta} dx dy\Big| \lesssim \| f \| _{L^p(\R^{n-1} )} \| g\| _{L^r(\R_+^n)}
\end{equation}
for any $f\in L^p(\R^{n-1})$ and any $g\in L^r(\R_+^n)$ together with
\[
\alpha < (n-1)(p-1)/p, \quad \beta<n(r-1)/r, \quad \alpha + \beta \geq 0,
\]
and 
\[
1/p + 1/r \geq 1
\]
and
the new balance condition
\begin{equation*}\label{eq-SW-BC-upper}
(n-1)/(n p) + 1 / r +(\alpha+\beta+\lambda)/n =2 - 1/n.
\end{equation*}
In the context of $\R^{n-k} \times \R^n$, it is expected that the following inequality holds
\begin{equation}\label{eq-MAIN-SW-k}
\Big| \iint_{\R^n \times \R^{n-k}} \frac{f(x) g(y)}{ |x|^\alpha |x-y|^\lambda |y|^\beta } dx dy \Big| 
\leq \M \| f \| _{L^p(\R^{n-k})} \| g\| _{L^r(\R^n)}
\end{equation}
for any functions $f\in L^p(\R^{n-k})$ and $g\in L^r(\R^n)$ under the following balance condition
\begin{equation}\label{eq-Identity-SW-k}
\frac{n-k}n \frac 1p + \frac 1r + \frac { \alpha + \beta + \lambda} n = 2 -\frac kn.
\end{equation}
We do not prove \eqref{eq-MAIN-SW-k} and leave it for future research. Instead, our aim is to replace the weight $|y|^{-\beta}$ by the weight $|y''|^{-\beta}$. In light of the discussion in subsection \ref{subsec-application-extended}, this gives an improvement of \eqref{eq-MAIN-SW-k} in the regime $\beta \geq 0$. 

Interestingly, our method of proving Theorem \ref{thm-MAIN} is flexible enough to obtain a similar result. To be more precise, we easily obtain the following SW type inequality on $\R^{n-k} \times \R^n$ with $0<\lambda<n-k$.

\begin{theorem}[SW type inequality on $\R^{n-k} \times \R^n$]\label{thm-MAIN-SW}
Let $n \geq 1$, $0\leq k <n$, $p, r>1$, $0<\lambda <n-k$, $0 \leq \alpha < (n-k)(p-1)/p$, and $\beta < k(r-1)/r$ satisfying 
\[
1/p + 1/r \geq 1
\]
and the balance condition \eqref{eq-Identity-SW-k}. Then there exists a sharp constant $\M > 0$ such that
\begin{subequations}\label{eq-MAIN-SW}
\begin{align}
\Big| \iint_{\R^n \times \R^{n-k}} \frac{f(x) g(y)}{ |x|^\alpha |x-y|^\lambda |y''|^\beta } dx dy \Big| 
\leq \M \| f \| _{L^p(\R^{n-k})} \| g\| _{L^r(\R^n)}
\tag*{ \eqref{eq-MAIN-SW}$_{k,\alpha, \beta}$}
\end{align}
\end{subequations}
for any functions $f\in L^p(\R^{n-k})$ and $g\in L^r(\R^n)$.
\end{theorem}

\begin{proof}
Notice that the balance condition \eqref{eq-Identity-SW-k} can be rewritten as follows
\[
\frac 1p + \frac 1r + \frac { \lambda - k/q  + \beta + \alpha} {n-k} = 2 .
\]
Now we can use Lemma \ref{lem-R2} with the function $h(x) = |x|^{-\alpha} f(x)$ and the classical SW inequality \eqref{eq-SW-whole} with $\beta = 0$ to get
\begin{align*}
\int_{\R^n} \Big( \int_{\R^{n-k}} \frac{f(x)}{ |x|^\alpha |x-y|^\lambda |y''|^\beta } dx \Big)^q dy 
& \overset{\eqref{eq-R1}} \lesssim \int_{\R^{n-k}} \Big( \int_{\R^{n-k}} \frac{f(x)}{ |x|^\alpha |x-y|^{\lambda - k/q  + \beta } } dx \Big)^q dy \\
& \overset{\eqref{eq-SW-whole}} \lesssim \| f \| _{L^p(\R^{n-k})}
\end{align*}
with $q=r/(r-1)$. Notice that in the regime $0<\lambda<n-k$, to be able to apply \eqref{eq-R1}, we require 
\[
\beta < k(r-1)/r, \quad 0<\lambda - k/q  + \beta < n-k.
\]
The latter inequality holds thanks to $\lambda > 0$, $\beta<k/q$, and $\alpha \geq 0$. Similarly, to be able to apply \eqref{eq-SW-whole} we need
\[
0\leq \alpha \leq (n-k)(p-1)/p, \quad 1/p+1/r \geq 1.
\] 
By duality, we obtain \eqref{eq-MAIN-SW}$_{k,\alpha, \beta}$ as expected. 
\end{proof}

The requirement $1/p + 1/r \geq 1$ is necessary for Theorem \ref{thm-MAIN-SW}. In the case $\alpha = 0$, it is automatically satisfied thanks to $\lambda - k/q + \beta \leq n-k$. Apparently, Theorem \ref{thm-MAIN-SW} can be improved for any $\lambda >0$, but we do not consider this case here and also leave it for future research.

Using the computation in subsection \ref{subsec-application-k->1} above, it is not hard to see that our SW type inequality \eqref{eq-MAIN-SW}$_{k,\alpha, \beta}$ also leads us to
\begin{equation}\label{eq-MAIN-SW-transformed}
\begin{aligned}
\iint_{\R^{n-k} \times \R_+^{n-k+1}} & \frac{f (x) G (y', \rho) d\rho dy' dx }{ |x|^\alpha \big[ |x-y'|^2+\rho^2 \big]^{\lambda/2} \rho^{\widehat\beta} } 
\lesssim \| f \| _{L^p(\R^{n-k})} \| G \| _{L^r (\R_+^{n-k+1})} ,
\end{aligned}
\end{equation}
which is an analogue of \eqref{eq-SW-upper}. Clearly, \eqref{eq-MAIN-SW-transformed} is also an improvement of \eqref{eq-SW-upper} if $\widehat\beta \geq 0$.



\subsection{Characterization of any optimal pair $(f^\sharp, g^\sharp)$}

We now turn our attention to an optimal pair $(f^\sharp, g^\sharp)$ for the maximizing problem \eqref{eq-VariationalProb} found by Proposition \ref{prop-EXISTENCE} above. To gain further properties on $(f^\sharp, g^\sharp)$, it is often to study the Euler--Lagrange equation for $\N $. Let us compute the Euler--Lagrange equation for $\F_{\lambda,k}$ together with the constraint 
\[
\|f \|_{L^p(\R^{n-k})} = \|g\|_{L^r(\R^n)}=1.
\] 
Clearly, with respect to the function $f$, the first variation of the functional $\F_{\lambda,k}$ is
\[
D_f (\F_{\lambda,k}) (f,g) (h) = \int_{\R^{n-k}} \Big( \int_{\R^n}\frac{ g(y) }{ |x-y|^\lambda |y''|^\beta }dy \Big) h(x) dx
\]
while the first variation of the constraint $\int_{\R^{n-k}} |f(x)|^p dx =1$ is
\[
p \int_{\R^{n-k}} |f(x)|^{p-2} f(x) h(x) dx.
\]
Therefore, by the Lagrange multiplier theorem, there exists some constant $\alpha$ such that
\[
\int_{\R^{n-k}} \Big( \int_{\R^n} \frac{g(y)}{ |x-y|^\lambda |y''|^\beta }dy \Big) h(x) dx = \alpha \int_{\R^{n-k}} |f(x)|^{p-2} f(x) h(x) dx
\]
holds for all function $h$ defined in $\R^{n-k}$. From this we know that $f$ and $g$ must satisfy the following equation
\[
\alpha |f(x)|^{p-2} f(x) 
= \int_{\R^n} \frac{ g(y)}{|x-y|^\lambda |y''|^\beta } dy .
\]
Interchanging the role of $f$ and $g$, we also know that $f$ and $g$ must fulfill the following
\[
\beta |g(y)|^{r-2} g(y) 
= \int_{\R^{n-k}} \frac{ f(x)}{|x-y|^\lambda |y''|^\beta } dx 
\]
for some new constant $\beta$. Using a suitable test function, we know that $\alpha = \beta = \F_{\lambda,k} (f,g)$. Hence, up to a constant multiple and simply using the following changes
\begin{equation*}\label{eq-(f,g)->(u,v)}
u =f^{p-1} \quad \text{and} \quad v = g^{r-1},
\end{equation*}
the two relations above lead us to the following system of integral equations
\begin{subequations}\label{eqIntegralSystem}
\begin{align}
\left\{
\begin{aligned}
u(x) &= \int_{\R^n} \frac{ v^\kappa (y) }{|x-y|^\lambda |y''|^\beta } dy,\\
v(y) &= \int_{\R^{n-k}} \frac{ u^\theta (x) }{|x-y|^\lambda |y''|^\beta } dx,
\end{aligned}
\right.
\tag*{\eqref{eqIntegralSystem}$_{k, \beta}$}.
\end{align}
\end{subequations}
with
\begin{equation*}\label{eq-(p,r)->(kappa,theta)}
\kappa = \frac 1{r-1} >0, \quad \theta = \frac 1{p-1} >0.
\end{equation*}
Using the balance condition \eqref{eq-Identity}, it is not hard to see that $\kappa$ and $\theta$ fulfill
\begin{equation}\label{eq-NecessaryCondition}
\frac{n - k}{n}\frac{1}{{\theta + 1}} + \frac{1}{\kappa + 1} = \frac{\lambda +\beta}{n}.
\end{equation}
In this sense, the condition \eqref{eq-NecessaryCondition} usually refers to the critical condition for \eqref{eqIntegralSystem}$_{k,\beta}$. From the above derivation, any optimal pair $(f^\sharp, g^\sharp)$ for the weighted HLS inequality \eqref{eq-MAIN}$_{k,\beta}$ must solve the system \eqref{eqIntegralSystem}$_{k,\beta}$. Hence, we can naively ask the following:
\begin{itemize}
 \item a regularity result for solutions to \eqref{eqIntegralSystem}$_{k,\beta}$ with $\lambda >0$ and
 \item a classification for solutions to \eqref{eqIntegralSystem}$_{k,\beta}$ with $\lambda >0$.
\end{itemize}  
As far as we know, there is no such a result for the above questions, even in the case $k=1$.


\section*{Acknowledgments}

This project initiated when the authors were visiting ICTP in 2017, whose support is greatly acknowledged. QAN is supported by the Tosio Kato Fellowship awarded in 2018. QHN is supported by the ShanghaiTech University startup fund. VHN would like to thank VIASM for support during the finalizing of the paper.


\appendix

\section{Proof of Lemma \ref{lem-FarAwayFromZero}}
\label{apd-FarAwayFromZero}

This appendix is devoted to the proof of Lemma \ref{lem-FarAwayFromZero}, namely we shall prove the following inequality
\begin{align*}
\int_{\R^n} (\EE [f]) ^q (y) dy \lesssim\| f\| _*^{q-p} \int_{\R^{n-k}} f^p (x)dx
\end{align*} 
for any $f \in L^p(\R^{n-k})$, which can be assumed to be non-negative. As mentioned earlier, Lemma \ref{lem-FarAwayFromZero} also provides us another way to prove \eqref{eq-MAIN}$_{k,\beta}$. To prove Lemma \ref{lem-FarAwayFromZero}, we mimic the proof of Adams for Riesz potentials; see \cite{Adams}. Recall from the definition of $\EE [f] $ the following
\begin{align*}
\EE [f] (y) &= \int_{\R^{n-k}}\frac{f(x)dx}{ |x-y|^\lambda |y''|^\beta} .
\end{align*}
Hence, our starting point is the following
\begin{equation}\label{eq-Head}
\begin{aligned}
\big\| \EE [f] \big\| _{L^q (\R^n)}^q 
&\lesssim
\int_{\R^{n-k}}\int_{0}^{+\infty}\left[\frac{\int_{B_\rho^{n-k} (y')} f(x) dx}{\rho^{\lambda-k/q+\beta}}\right]^q \frac{d\rho}{\rho} dy',
\end{aligned}
\end{equation}
thanks to \eqref{eqKeyKey}. In the sequel, we prove that
\begin{align*}
 \int_{\R^{n-k}}\int_{0}^{+\infty}\left[\frac{\int_{B_\rho^{n-k} (y')} f(x) dx}{\rho^{\lambda -k/p+\beta}}\right]^q \frac{d\rho}{\rho} dy'\lesssim \| f \| _*^{q-p} \| f\| _{L^p(\R^{n-k})}^p.
\end{align*}
To this purpose, recall the following estimate
	\begin{align*}
	\int_{B_\rho^{n-k} (y')} f(x) dx\lesssim \rho^{n-k} (Mf) (y'),
	\end{align*}
where $Mf$ is the Hardy--Littlewood maximal function of $f \geq 0$ in $\R^{n-k}$, defined by
\[
Mf (z) = \sup_{r>0} \frac 1{|B^{n-k}(z,r)|} \int_{|x-z| \leq r} f(x) dx.
\]
In addition, from the definition, we also have
\[
 \int_{B_\rho^{n-k} (y')} f(x) dx \lesssim \rho^{ \frac{n-k}{p'}} \| f \| _* 
\] 
for any $\rho>0$ and any $y' \in \R^{n-k}$. Thus, for some $\delta > 0$ to be determined later, we can estimate
\begin{align*}
\int_{0}^{+\infty}\Big[ & \frac{\int_{B_\rho^{n-k} (y')} f(x) dx}{\rho^{\lambda -k/q + \beta}}\Big]^q\frac{d\rho}{\rho} \\
&=\Big( \int_{0}^{\delta}+\int_{\delta}^{+\infty} \Big) \Big[\frac{\int_{B_\rho^{n-k} (y')} f(x) dx}{\rho^{\lambda -k/q+ \beta}}\Big]^q \frac{d\rho}{\rho} \\
&\lesssim \int_{0}^{\delta} \Big[\frac{ \rho^{n-k} (Mf) (y')}{\rho^{\lambda -k/q +\beta}}\Big]^q \frac{d\rho}{\rho} 
+\int_{\delta}^{+\infty} \Big[\frac{\rho^{(n-k)/p'} \| f\| _*}{\rho^{\lambda -k/q+ \beta}}\Big]^q \frac{d\rho}{\rho} \\
&\lesssim \delta^{(n-\lambda- \beta) q-(q-1)k} f^*(y')^q 
 + \delta^{(n-\lambda- \beta) q-(q-1)k-\frac{q(n-k)}{p}}\| f\| _*^q.
\end{align*}
To obtain the last line in the above estimate, we also note that
\begin{equation*}\label{eq-ChooseQ1}
\lambda- \frac kq + \beta - \frac{n-k}{p'} > 0,
\end{equation*}
thanks to \eqref{eq-Identity}; otherwise, the integral $\int_\delta^{+\infty}$ diverges. The trick is first to select $\delta$ such that
\begin{equation}\label{eq-ChooseDelta}
\delta^{(n-\lambda- \beta) q-(q-1)k} (Mf) (y')^q =\| f\| _*^{q-p} (Mf) (y')^p
\end{equation}
and then to select $q$ such that
\begin{equation}\label{eq-ChooseQ2}
\delta^{(n-\lambda- \beta) q-(q-1)k-\frac{q(n-k)}{p}}\| f\| _*^q =\| f\| _*^{q-p} (Mf) (y')^p.
\end{equation}
Indeed, to fulfill \eqref{eq-ChooseDelta}, we simply choose
\[
\delta=\Big[\frac{\| f\| _* }{Mf (y')}\Big]^{\frac{q-p}{(n-\lambda- \beta) q-(q-1)k}}.
\] 
From this choice of $\delta$ we deduce 
\[
\delta^{(n-\lambda- \beta) q-(q-1)k} 
=\Big[\frac{\| f\| _* }{Mf (y')}\Big]^{q-p},
\]
which immediately implies \eqref{eq-ChooseDelta}. Notice that
\begin{align*}
(n-\lambda- \beta) q-(q-1) & k-\frac{q(n-k)}{p}\\
=&\frac q{q-p} \big[(n-\lambda- \beta) q-(q-1)k \big] -\frac{q(n-k)}{p} \\
&-\frac p{q-p} \big[(n-\lambda- \beta)q - (q-1)k \big]\\
=& -\frac p{q-p} \big[(n-\lambda- \beta)q - (q-1)k \big],
\end{align*}
thanks to \eqref{eq-Identity}. Hence, we obtain
\begin{align*}
\delta^{(n-\lambda- \beta) q-(q-1)k-\frac{q(n-k)}{p}}
=\Big[\frac{\| f\| _* }{Mf (y')}\Big]^{-p},
\end{align*}
which yields
\[
\delta^{(n-\lambda) q-(q-1)k-\frac{q(N-k)}{p}}\| f\| _*^q
=\| f\| _*^{q-p} (Mf) (y')^p,
\]
which is nothing but \eqref{eq-ChooseQ2}. Thus, we arrive at
\begin{align*}
\int_{0}^{+\infty}\left[\frac{\int_{B_\rho^{n-k} (y')}|f(x)| dx}{\rho^{\lambda -k/q + \beta}}\right]^q\frac{d\rho}{\rho} \lesssim \| f\| _*^{q-p} (Mf) (y')^p.
\end{align*}
Since $\| f^*\| _{L^p(\R^{n-k})} \lesssim \| f\| _{L^p(\R^{n-k})}$, we deduce that
\begin{equation}\label{eq-Tail}
\begin{aligned}
\int_{\R^{n-k}}\int_{0}^{+\infty}\left[\frac{\int_{B_\rho^{n-k} (y')}|f(x)| dx}{\rho^{\lambda -k/q + \beta}}\right]^q \frac{d\rho}{\rho} dy'
&\lesssim \| f\| _*^{q-p}\int_{\R^{n-k}} (Mf) (y')^p dy' .
\end{aligned}
\end{equation}
Combining \eqref{eq-Head} and \eqref{eq-Tail} gives the desired estimate.


\end{document}